\documentclass{amsart}

\usepackage{amsmath,color,amssymb,graphics,latexsym}

\newcommand{\tr}{{\rm tr}\,}

\newcommand{\bmat}{\left[ \begin{array}}
\newcommand{\emat}{\end{array} \right]}
\newcommand{\ignore}[1]{}

\newtheorem{lemmaA}{Lemma}
\newtheorem{lemmaB}{Lemma}

\theoremstyle{definition}

\newtheorem*{unprop}{Proposition}

\newtheorem*{unconjecture}{Conjecture}

\theoremstyle{remark}

\numberwithin{equation}{section}

\begin{document}

\title{Hypergeometric Functions I}

\author{Ian G. Macdonald}

\address{56 High Street, Steventon, Oxfordshire OX13 6RS, England} 

\date{\today}


\maketitle

\tableofcontents

\section*{Foreword}
This is the typewritten version of a handwritten manuscript which was
completed by Ian G. Macdonald in 1987 or 1988. The manuscript is a very
informal working paper, never intended for formal publication.
Nevertheless, copies of the manuscript have circulated widely, giving rise
to quite a few citations in the subsequent 25 years. Therefore it seems
justified to make the manuscript available for the whole mathematical
community. The author kindly gave his permission that a typewritten version be posted on arXiv.
These notes were typeset verbatim by Tierney Genoar and Plamen Koev, supported by the San Jose State University Planning Council and National Science Foundation Grant DMS-1016086.
The manuscript is followed by ``Hypergeometric functions II ($q$-analogues).''
\newpage

\section{}
Hypergeometric functions $_pF_q$ on the space $\Sigma_n$ of real $n\times n$ 
symmetric matrices were introduced by Herz \cite{Herz55}. His definition was 
inductive; he started from
$$
_0F_0(s)=\exp(\tr(s)), \quad (s\in \Sigma_n)
$$
and used a Laplace transform (resp.\ inverse Laplace transform) to pass from 
$_pF_q$ to $_{p+1}F_q$ (resp.\ to $_pF_{q+1}$). Subsequently, Constantine \cite{Constantine63} showed that $_pF_q$ could be expanded naturally as a series of zonal 
polynomials, and we shall take this series as our definition.

Let $C_\lambda$ ($\lambda$ a partition of length $\le n$) denote the zonal 
polynomial indexed by $\lambda$, normalized so that
\begin{equation}
\sum_{\lambda\vdash m} C_\lambda=p_1^m
\label{eq_1.1}
\end{equation}
for all $m\ge 0$. Since the coefficient of $p_1^m$ in 
$J_\lambda=J_\lambda(x;2) (=Z_\lambda(x))$ is 1 for all $\lambda\vdash m$, 
we have
$$
\langle J_\lambda,C_\lambda \rangle_2=
\langle J_\lambda,p_1^m \rangle_2=
\langle p_1^m,p_1^m \rangle_2=2^mm!
$$
so that
\begin{equation}
C_\lambda=2^mm! J_\lambda^*,
\label{eq_1.2}
\end{equation}
where ($J_\lambda^*$) is the basis of $\Lambda$ dual to $(J_\lambda)$, i.e., 
$J_\lambda^*=J_\lambda/|J_\lambda|_2^2$.

If $\lambda=(\lambda_1,\ldots,\lambda_n)$ is a partition of length $\le 
n$, define
\begin{equation}
(a)_\lambda=\prod_{i=1}^n (a-\tfrac{1}{2}(i-1))_{\lambda_i}
\end{equation}
and if $\underline{a}=(a_1,\ldots,a_p)$, define
\begin{equation}
(\underline{a})_\lambda=(a_1)_\lambda\cdots (a_p)_\lambda.
\end{equation}

With this notation established, we define
\begin{eqnarray}
_pF_q(\underline{a};\underline{b};s)&=&\sum_\lambda 
\frac{(\underline{a})_\lambda}{(\underline{b})_\lambda}\cdot
\frac{C_\lambda(s)}{|\lambda|!}
\label{eq_1.5}
\\
_pF_q(\underline{a};\underline{b};s,t)&=&\sum_\lambda 
\frac{(\underline{a})_\lambda}{(\underline{b})_\lambda}\cdot
\frac{C_\lambda(s)C_\lambda(t)}{C_\lambda(1_n)|\lambda|!}.
\label{eq_1.6}
\end{eqnarray}
Here $\underline{a}=(a_1,\ldots,a_p)$ and 
$\underline{b}=(b_1,\ldots,b_q)$ are sequences of lengths $p, q$, 
respectively, and $s,t\in \Sigma_n$.

The functions $C_\lambda$ on $\Sigma_n$ are invariant under the action of 
$K=O(n)$, i.e., 
$$
C_\lambda(s)=C_\lambda(ksk')
$$
for $s\in\Sigma_n$ and $k\in K$.

If $s,t\in \Sigma_n^+$, the cone of positive definite matrices in 
$\Sigma_n$, $s$ and $t$ have unique positive square roots, $s^{1/2}, 
t^{1/2}$ (reduce to diagonal form and take the positive square roots of 
the eigenvalues), and we define
$$
C_\lambda(st)=C_\lambda(s^{1/2}ts^{1/2})=C_\lambda(t^{1/2}st^{1/2}).
$$
The doubling principle for zonal spherical functions then gives
\begin{equation}
\int_K C_\lambda(sktk') \mathrm{d} k = 
\frac{C_\lambda(s)C_\lambda(t)}{C_\lambda(1_n)},
\label{eq_1.7}
\end{equation}
where d$k$ is normalized Haar measure on $K=O(n)$.

From \eqref{eq_1.7} and the definitions \eqref{eq_1.5}, \eqref{eq_1.6} it 
follows that 
\begin{equation}
\int_K \mbox{}_pF_q(\underline{a};\underline{b};sktk') \mathrm{d}k =
\mbox{}_pF_q(\underline{a};\underline{b};s,t).
\label{eq_1.8}
\end{equation}

\underline{Remark}. In view of \eqref{eq_1.2} we can rewrite the 
definitions \eqref{eq_1.5}, \eqref{eq_1.6} in terms of $J_\lambda^*$:
\begin{align}
_pF_q(\underline{a};\underline{b};s)&=\sum_\lambda 
\frac{(\underline{a})_\lambda}{(\underline{b})_\lambda} 2^{|\lambda|}J^*_\lambda(s),
\label{eq_1.5p}
\tag{\ref{eq_1.5}$'$}
\\
_pF_q(\underline{a};\underline{b};s,t)&=\sum_\lambda 
\frac{(\underline{a})_\lambda}{(\underline{b})_\lambda} 2^{|\lambda|}
\frac{J^*_\lambda(s)J^*_\lambda(t)}{J^*_\lambda(1_n)}.
\label{eq_1.6p}
\tag{\ref{eq_1.6}$'$}
\end{align}

Notice also that 
\begin{equation}
_pF_q(\underline{a};\underline{b};s,1)=
\mbox{}_pF_q(\underline{a};\underline{b};s).
\label{eq_1.9}
\end{equation}
\newpage

\section{Particular cases}
\begin{equation}
_0F_0(s)=\exp(\tr s), \quad (s\in \Sigma_n).
\label{eq_2.1}
\end{equation}
\begin{proof}
From the definition and \eqref{eq_1.1} we have
$$
_0F_0(s)=\sum_\lambda \frac{C_\lambda(s)}{|\lambda|!}=
\sum_{m\ge 0} \frac{({\rm trace}\; s)^m}{m!}=e^{{\rm trace}(s)}.
$$
\end{proof}

\begin{equation}
_1F_0(a;s)=|1-s|^{-a},
\label{eq_2.2}
\end{equation}
where $|1-s|=\det(1-s)$.
\begin{proof}
From the definition \eqref{eq_1.5} we have
\begin{eqnarray*}
_1F_0(a;s) &=& \sum_\lambda\frac{(a)_\lambda C_\lambda(s)}{|\lambda|!} \\
&=& \sum_\lambda 2^{|\lambda|}(a)_\lambda J_\lambda^*(s)
\end{eqnarray*}
by \eqref{eq_1.2}. Now if $\varepsilon_X$ is the specialization $p_r 
\mapsto X$ (all $r\ge 1$) (so that $\varepsilon_n(f)=f(1_n)$) we have
\begin{eqnarray*}
\varepsilon_X(J_\lambda) & = & 
\prod_{(i,j)\in\lambda} (X+2(j-1)-(i-1))
\\
&=&2^{|\lambda|}\prod_{(i,j)\in\lambda} (\tfrac{1}{2}(X-i+1)+(j-1))
\\
&=&2^{|\lambda|}\prod_{i=1}^n 
(\tfrac{1}{2}X-\tfrac{1}{2}(i-1))_{\lambda_i},
\end{eqnarray*}
i.e.,
\begin{equation}
\varepsilon_X(J_\lambda)=2^{|\lambda|}(\tfrac{1}{2}X)_\lambda
\label{eq_2.3}
\end{equation}
and therefore (if $x={\rm diag}(x_1,\ldots,x_n)$ is a diagonal matrix)
\begin{eqnarray*}
_1F_0(a;x) 
&=& \varepsilon_{2a}^{(y)}\sum_\lambda J_\lambda(y) J_\lambda^*(x)
\\
&=& \varepsilon_{2a}^{(y)}\prod_{i,j=1}^n(1-x_iy_j)^{-1/2}
\\
&=& \prod_{i=1}^n(1-x_i)^{-a}
\\
&=& |1-x|^{-a}.
\end{eqnarray*}
\end{proof}

\begin{equation}
_0F_1(\tfrac{1}{2}r;s)=\int_K \exp \tr(2xk) \mathrm{d}k. \quad (s=xx')
\label{eq_2.4}
\end{equation}

\begin{proof}
We have
$$
(\tr(2xk))^m = \sum_{\mu\vdash m} \chi_{(1^m)}^\mu s_\mu(2xk)
$$
so that
\begin{eqnarray*}
\int_K \exp(\tr(2xk))\mathrm{d}k 
&=& \sum_{m\ge 0} \frac{1}{m!}
\sum_{\mu\vdash m} \chi_{(1^m)}^\mu 2^{|\mu|} \int_K s_\mu(xk) \mathrm{d}k
\\
&=& \sum_{m\ge 0} \frac{2^{2m}}{(2m)!}\sum_{\lambda\vdash m} 
\chi_{(1^{2m})}^{2\lambda} \Omega_\lambda(x),
\end{eqnarray*}
where
$\Omega_\lambda(x)=J_\lambda(xx')/J_\lambda(1_n)$ is the normalized 
spherical function (on $G/K$).

Now 
$$
\frac{1}{(2m)!} \chi_{(1^{2m})}^{2\lambda} = \frac{1}{h(2\lambda)},
$$
where $h(2\lambda)$ is the product of the hook lengths of $2\lambda$, and
$$
h(2\lambda)=|J_\lambda|_2^2.
$$
Hence, if $s=xx'$, we have
$$
\int_K \exp(\tr 2xk) \mathrm{d}k = \sum_\lambda 
\frac{2^{2|\lambda|}J_\lambda^*(s)}{J_\lambda(1_n)}
=\sum_\lambda
\frac{2^{|\lambda|}J_\lambda^*(s)}{(\frac{1}{2}n)_\lambda}
$$
by \eqref{eq_2.3}. By \eqref{eq_1.5p} this last sum is equal to 
$_0F_1(\frac{1}{2}n;s)$.
\end{proof}

We remark that $_0F_1(b;s)$ is essentially a generalized Bessel function 
(Herz \cite{Herz55}).

When $n=1$, \eqref{eq_2.3} gives 
$$
_0F_1(\tfrac{1}{2};x^2)=\sum_{r\ge 0} \frac{x^{2r}}{(\tfrac{1}{2})_rr!}
=\sum_{r\ge 0} \frac{(2x)^{2r}}{(2r)!}={\rm ch}\, 2x
$$
and (since $K=O(1)=\{\pm 1\}$)
$$
\int_K e^{2xk} \mathrm{d}k = \frac{1}{2}(e^{2x}+e^{-2x}) = {\rm ch}\, 2x.
$$

From \eqref{eq_1.8} and \eqref{eq_2.1}, \eqref{eq_2.2} we have
\begin{eqnarray}
_0F_0(s,t)&=&\int_K \exp \tr(sktk') \mathrm{d}k,
\label{eq_2.5}
\\
_1F_0(a;s,t)&=&\int_K |1-sktk'|^{-a} \mathrm{d}k.
\label{eq_2.6}
\end{eqnarray}

Hence
\begin{eqnarray*}
_0F_0(s,1+t)&=&\int_K \exp \tr (sk(1+t)k')\mathrm{d}k
\\
&=&
e^{\tr(s)}\int_K \exp \tr (sktk')\mathrm{d}k,
\end{eqnarray*}
i.e., 
\begin{equation}
_0F_0(s,1+t)=e^{\tr(s)}\mbox{}_0F_0(s,t).
\label{eq_2.7}
\end{equation}

Again,
\begin{eqnarray*}
_1F_0(a;s,1+t)&=&\int_K |1-sk(1+t)k'|^{-a} \mathrm{d}k
\\
&=&
|1-s|^{-a}\int_K \left|1-\frac{s}{1-s}ktk'\right|^{-a} \mathrm{d}k
\end{eqnarray*}
so that 
\begin{equation}
_1F_0(a;s,1+t)=\mbox{}_1F_0(a;s)\mbox{ } _1F_0(a; \tfrac{s}{1-s},t).
\label{eq_2.8}
\end{equation}
\newpage

\section{Integral formulae}

As before, let $\Sigma_n^+$ (or just $\Sigma^+$) denote the cone of 
positive definite $n\times n$ real symmetric matrices. We take as measure 
on $\Sigma^+$
\begin{equation}
\mathrm{d}s = c_n \prod_{i\le j}\mathrm{d}s_{ij},
\label{eq_3.1}
\end{equation}
where $s=(s_{ij})_{1\le i,j\le n}$ and 
$c_n=\pi^{-n(n-1)/4}$. (This constant is built into the measure $\mathrm{d}s$ in 
order to prevent it appearing everywhere else.) Also define
\begin{equation}
\Gamma_n(a) = \prod_{i=1}^n \Gamma(a-\tfrac{1}{2}(i-1)).
\label{eq_3.2}
\end{equation}

The basic integral formula, from which all else follows, is then
\begin{equation}
\int_{\Sigma^+} e^{-\tr(st)} |s|^{a-p} \mathrm{d}s = |t|^{-a} \Gamma_n(a),
\label{eq_3.3}
\end{equation}
where $p=\frac{1}{2}(n+1)$ and $|s|=\det(s)$.

\begin{proof}
We first reduce to the case $t=1_n$. Let 
$$
\mathrm{d}^*s = |s|^{-p} \mathrm{d}s,
$$
then $\mathrm{d}^*s$ is a $G$-invariant measure on $\Sigma$, where $G={\rm 
GL}_n(\mathbb{R})$. That is to say, we have
$$
\mathrm{d}^*(xsx')=\mathrm{d}^*x \quad (x\in G, s\in \Sigma).
$$

Take $x=t^{1/2}$, the positive square root of $t$. Then the left hand side 
of \eqref{eq_3.3} is 
\begin{eqnarray*}
\int_{\Sigma^+} e^{-\tr(xsx')} |s|^a \mathrm{d}^*s
&=&
\int_{\Sigma^+} e^{-\tr s} |x^{-1}s(x^{-1})'|^a \mathrm{d}^*s \\
&=&|t|^{-a}\int_{\Sigma^+} e^{-\tr(s)}|s|^{a-p} \mathrm{d}s.
\end{eqnarray*}
We have to evaluate this integral. For this purpose we invoke rational 
reduction of quadratic forms to write $s=y'y$ (for almost all $s\in 
\Sigma^+$) with $y\in G$ upper triangular, say $y=ux$ where $u=(u_{ij})$ 
is upper unitriangular and 
$x={\rm diag}(x_1,\ldots,x_n)$ has positive entries $x_i$. Let
$$
\mathrm{d}x=\prod_{i=1}^n \mathrm{d}x_i, \quad  \mathrm{d}u = \prod_{i<j} \mathrm{d}u_{ij};
$$
then the decomposition $s=(ux)'(ux)$ gives
$$
\mathrm{d}^* s = 2^n c_n |x|^{-1} \mathrm{d}x\mathrm{d}u
$$
so that
$$
|s|^{a-p} \mathrm{d}s= 2^n c_n |x|^{2a-1} \mathrm{d}x\mathrm{d}u.
$$

We have 
$$
{\rm trace}(s)=\sum_{k\le i} x_i^2 u_{ki}^2
$$
and hence
$$
\int_{\Sigma^+}
e^{-\tr(s)}|s|^{a-p}\mathrm{d}s = 
2^nc_n \int \exp(-\sum x_i^2u_{ki}^2)\cdot |x|^{2a-1}\mathrm{d}x\mathrm{d}u,
$$
wherein each $u_{ij}$ (resp.\ $x_i$) is integrated over $\mathbb R$
(resp.\ $\mathbb R^+$).

Integrate first over $U$: Since
$$
\int_{-\infty}^\infty e^{-x^2u^2} \mathrm{d}u = \pi^{1/2} x^{-1}
$$
we obtain (since $c_n=\pi^{-n(n-1)/4}$)
\begin{eqnarray*}
2^n\left(\int_0^\infty \right)^n e^{-\sum x_i^2}
\prod_{k<i} x_i^{-1} \prod_{i=1}^n x_i^{2a-1} \mathrm{d}x
&=&
\prod_{i=1}^n \left(2\int_0^\infty e^{-x_i^2}x_i^{2a-i} \mathrm{d}x_i \right)
\\
&=&
\prod_{i=1}^n \left(\int_0^\infty e^{-x}x^{a-\frac{1}{2}(i-1)} 
\frac{\mathrm{d}x}{x} \right)
\\
&=&
\prod_{i=1}^n \Gamma(a-\tfrac{1}{2}(i-1)) = \Gamma_n(a).
\end{eqnarray*}
\end{proof}

Next, replace $t$ by $1_n-t$ in \eqref{eq_3.3}, where $0<t<1_n$ (for the 
partial order on $\Sigma$ defined by 
$$
s\le t\quad\mbox{ iff $t-s$ is positive semidefinite)}.
$$

We have then
\begin{equation}
\int_{\Sigma^+} e^{-\tr(s(1-t))} |s|^{a-p} \mathrm{d}s = |1-t|^{-a} \Gamma_n(a).
\label{eq_3.4}
\end{equation}
But, on the other hand,
\begin{eqnarray*}
\int_{\Sigma^+} e^{-\tr(s(1-t))} |s|^{a-p} \mathrm{d}s &=&
\int_{\Sigma^+} e^{-\tr(s)} |s|^{a-p} \left(\int_K e^{\tr(ksk't)} \mathrm{d}k 
\right)\mathrm{d}s 
\\
&&\mbox{(by replacing $s$ by $ksk'$ and integrating over $K$)}
\\
&=&\int_{\Sigma^+} e^{-\tr(s)} |s|^{a-p} \left(\int_K 
\sum_\lambda\frac{C_\lambda(ksk't)}{|\lambda|!} \mathrm{d}k 
\right)\mathrm{d}s 
\\
&=&\int_{\Sigma^+} e^{-\tr(s)} |s|^{a-p} \left(\sum_\lambda 
\Omega_\lambda(s) \frac{C_\lambda(t)}{|\lambda|!} 
\right)\mathrm{d}s.
\end{eqnarray*}
Moreover, we have
$$
(1-t)^{-a} = \sum_\lambda (a)_\lambda \frac{C_\lambda(t)}{|\lambda|!}
$$
so that from \eqref{eq_3.4} and these calculations we obtain
$$
\int_{\Sigma^+} e^{-\tr(s)} |s|^{a-p} \Omega_\lambda(s) 
\mathrm{d}s = (a)_\lambda \Gamma_n(a), 
$$
or, if we define
\begin{equation}
\Gamma_n(a;\lambda) = \prod_{i=1}^n \Gamma(a+\lambda_i-\tfrac{1}{2}(i-1)),
= (a)_\lambda \Gamma_n(a)
\label{eq_3.5}
\end{equation}
we have proved that
\begin{equation}
\int_{\Sigma^+}e^{-\tr(s)}|s|^{a-p}\Omega_\lambda(s)\mathrm{d}s = \Gamma_n(a;\lambda)
\label{eq_3.6}
\end{equation}
for any partition $\lambda$ of length $\le n$.

More generally, we have
\begin{equation}
\int_{\Sigma^+} e^{-\tr(st)} |s|^{a-p} \Omega_\lambda(s) \mathrm{d}s
=|t|^{-a} \Gamma_n(a;\lambda) \Omega_\lambda(t^{-1}).
\label{eq_3.7}
\end{equation}

\begin{proof}
Replace $s$ by $t^{1/2}st^{1/2}$ as in the proof of \eqref{eq_3.3}, \&
then use the doubling principle \eqref{eq_1.7} to replace 
$\Omega_\lambda(st^{-1})$ in the integrand by 
$\Omega_\lambda(s)\Omega_\lambda(t^{-1})$.
\end{proof}

An equivalent version of \eqref{eq_3.7} is 
\begin{equation}
\int_{\Sigma^+} e^{-\tr s}|s|^{a-p} \Omega_\lambda(st) \mathrm{d}s = 
\Gamma_n(a;\lambda) \Omega_\lambda(t).
\label{eq_3.7p}
\tag{\ref{eq_3.7}$'$}
\end{equation}

\begin{proof}
Replace $(s,t)$ in \eqref{eq_3.7} by $(st,t^{-1})$.
\end{proof}

\eqref{eq_3.7}: Laplace transform of $|s|^{a-p}\Omega_\lambda(s)$ is 
$|t|^{-a} \Gamma_n(a;\lambda) \Omega_\lambda(t^{-1})$.

\subsection*{Laplace transform}
The Laplace transform $Lf$ of a function $f$ on $\Sigma^+$ is defined by
\begin{equation}
(Lf)(t)=\int_{\Sigma^+} e^{-\tr(st)} f(s) \mathrm{d}s.
\label{eq_3.8}
\end{equation}

With this notation, \eqref{eq_3.7} can be restated as follows.
\begin{flalign}
\mbox{If }f(s)=|s|^{a-p}\Omega_\lambda(s),\mbox{ then }&&
\label{eq_3.7pp}
\tag{\ref{eq_3.7}$''$}
\end{flalign}
$$
(Lf)(t)=|t|^{-a} \Gamma_n(a;\lambda)\Omega_\lambda(t^{-1}).
$$

\begin{flalign}
\mbox{Let}&&
\label{eq_3.9}
\end{flalign}
$$
f(t)=\int_0^t f_1(s)f_2(t-s) \mathrm{d}s\quad (t\in \Sigma^+)
$$
where the integration is over $[0,t]=\{s: 0<s<t\}$ in $\Sigma_n^+$. Then 
$$
Lf=(Lf_1)(Lf_2)
$$
(convolution theorem).

\begin{proof}
By definition,
$$
(Lf)(u)=\int_{\Sigma^+} e^{-\tr(tu)} \left(
\int_0^t f_1(s) f_2(t-s) \mathrm{d}s
\right)\mathrm{d}t.
$$
Put $t=s+s_1$, then this becomes (Fubini)
$$
\left(\int_{\Sigma^+} e^{-\tr(su)}f_1(s)\mathrm{d}s\right)
\left(\int_{\Sigma^+} e^{-\tr(s_1u)}f_2(s_1)\mathrm{d}s_1\right),
$$
since $\tr(tu)=\tr(s+s_1)u=\tr(su)+\tr(s_1u)$, and $\mathrm{d}s\mathrm{d}t=\mathrm{d}s\mathrm{d}s_1$.
\end{proof}

As an application of \eqref{eq_3.9}, let us take 
$$
f_1(s)=|s|^{a-p}\Omega_\lambda(s), \quad f_2(s)=|s|^{b-p}.
$$

Then by \eqref{eq_3.7pp} we have
\begin{eqnarray*}
(Lf_1)(t) & =& \Gamma_n(a;\lambda) |t|^{-a} \Omega_\lambda(t^{-1}),
\\
(Lf_2)(t) & =& \Gamma_n(b) |t|^{-b},
\end{eqnarray*}
and
\begin{eqnarray*}
f(t)    & = & \int _0^t \Omega_\lambda(s)|s|^{a-p} |t-s|^{b-p} \mathrm{d}s,
\\
(Lf)(t) & = & \Gamma_n(a;\lambda)\Gamma_n(b)|t|^{-(a+b)} \Omega_\lambda(t^{-1}),
\end{eqnarray*}
which by \eqref{eq_3.7pp} is the Laplace transform of
$$
\frac{\Gamma_n(a;\lambda)\Gamma_n(b)}{\Gamma_n(a+b;\lambda)}
\Omega_\lambda(t).
$$

Setting $t=1_n$, we obtain (by uniqueness of Laplace transform)

\begin{equation} 
\int_0^{1_n}
\Omega_\lambda(s)|s|^{a-p}|1-s|^{b-p} \mathrm{d}s
=
\frac{\Gamma_n(a;\lambda)\Gamma_n(b)}{\Gamma_n(a+b;\lambda)}
=
B_n(a,b)(a)_\lambda/(a+b)_\lambda,
\label{eq_3.10} 
\end{equation}
where
\begin{equation} 
B_n(a,b)=
\frac{\Gamma_n(a)\Gamma_n(b)}{\Gamma_n(a+b)}
=\int_0^{1_n} |s|^{a-p} |1-s|^{b-p} \mathrm{d}s.
\label{eq_3.11} 
\end{equation}

The interval $[0,1_n]$ in $\Sigma^+$ is $K$-stable, because $s$ and $ksk'$ 
have the same eigenvalues. From \eqref{eq_3.10} we deduce that
\begin{equation} 
\int_0^{1_n} \Omega_\lambda(st)|s|^{a-p} |1-s|^{b-p} \mathrm{d}s
= B_n(a,b)\frac{(a)_\lambda}{(a+b)_\lambda} \Omega_\lambda(t).
\label{eq_3.12} 
\end{equation}
\begin{proof}
Replace $s$ by $ksk'$ ($k\in K$) in the integrand, and then integrate over 
$K$ \& use the doubling principle \eqref{eq_1.7}. 
\end{proof}

We now apply these formulas to hypergeometric functions.

Let
\begin{alignat*}{2}
\underline{a} & =   (a_1,\ldots,a_p), &\qquad
\underline{a}^+  & =   (a_1,\ldots,a_p,a),
\\
\underline{b} & =  (b_1,\ldots,b_q), & \qquad
\underline{b}^+ & = (b_1,\ldots,b_q,b).
\end{alignat*}

\begin{equation} 
\int_{\Sigma^+} e^{-\tr(s)} \mbox{} _pF_q (\underline{a};\underline{b};
st)|s|^{a-p} \mathrm{d}s
=\Gamma_n(a) \, _{p+1}F_q(\underline{a}^+;\underline{b};t).
\label{eq_3.13} 
\end{equation}
\begin{proof}
This proof follows from \eqref{eq_3.7p} if we replace $\Omega_\lambda$ 
there by $J_\lambda^*$ on either side.
\end{proof}

Equivalently,
\begin{flalign} 
\mbox{The function }f(s)=|s|^{a-p} 
\mbox{}_pF_q(\underline{a};\underline{b};s) &&
\label{eq_3.13p} 
\tag{\ref{eq_3.13}$'$}
\end{flalign}
has Laplace transform $Lf(t)$, where
$$
(Lf)(t^{-1}) = 
\Gamma_n(a)|t|^a\,_{p+1}F_q(\underline{a}^+;\underline{b};t).
$$

Next, from \eqref{eq_3.12} we have\footnote{See Additional observation at the end of this section.}
\begin{equation} 
\int_0^{1_n} \mbox{} _pF_q (\underline{a};\underline{b};
st)|s|^{a-p} |1-s|^{b-a-p} \mathrm{d}s 
=\frac{\Gamma_n(a)\Gamma_n(b-a)}{\Gamma_n(b)}
\,_{p+1}F_{q+1}(\underline{a}^+;\underline{b}^+;t).
\label{eq_3.14} 
\end{equation}

A particular case of \eqref{eq_3.14} is ($p=1, q=0$):
\begin{equation}
_2F_1(a,b;c;t)=\frac{\Gamma_n(c)}{\Gamma_n(a)\Gamma_n(c-a)}
\int_0^{1_n} \frac{|s|^{a-p} |1-s|^{c-a-p} }{|1-st|^b}\mathrm{d}s
\label{eq_3.15} 
\end{equation}
(Herz \cite[ (2.12)]{Herz55}). Another particular case is
\begin{equation}
_1F_1(a;b;t)=\frac{\Gamma_n(b)}{\Gamma_n(a)\Gamma_n(b-a)}
\int_0^{1_n} e^{\tr(st)}|s|^{a-p} |1-s|^{b-a-p} \mathrm{d}s.
\label{eq_3.16} 
\end{equation}

In \eqref{eq_3.15}, replace $s$ by $1-s$; we obtain
$$
_2F_1(a,b;c;t)=\frac{\Gamma_n(c)}{\Gamma_n(a)\Gamma_n(c-a)}
\int_0^{1_n} \frac{|s|^{c-a-p} |1-s|^{a-p} }{|1-t+st|^b}\mathrm{d}s
=|1-t|^{-b}\mbox{}
_2F_1(c-a,b;c;-t(1-t)^{-1}).
$$

If we now reiterate this, with $b,c-a$ taking the roles of $a,b$ and 
observe that $|1+t(1-t)^{-1}|=|1-t|^{-1}$, we obtain Euler's relation
\begin{equation}
_2F_1(a,b;c;t)=
|1-t|^{c-a-b}\mbox{}_2F_1(c-a,c-b;c;t).
\label{eq_3.17}
\end{equation}

Again, by replacing $s$ by $1-s$ in \eqref{eq_3.16}, we obtain (Kummer)

\begin{equation}
_1F_1(a;b;t)=
e^{\tr(t)}\mbox{}_1F_1(b-a;b;-t).
\label{eq_3.18} 
\end{equation}
\bigskip

\subsection*{Additional observation}

Inverse Laplace transform
$$_pF_{q+1}(\underline{a},\underline{b}^+;s) = \frac{\Gamma_n(b)}{(2\pi)^ni^{n(n+1)/2}}\int_{\mathrm{Re}(t)=x_0>0}e^{\mathrm{tr}(t)}\mbox{}_pF_q(\underline{a};\underline{b};t^{-1}s)'|t|^{-b}\mathrm{d}t.$$
$t\in\Sigma^+ \oplus i\Sigma$
\newpage

\section{Gauss \& Saalschutz summation}

If we set $t=1_n$ in \eqref{eq_3.15} we obtain
$$
_2F_1(a,b;c;1_n)=\frac{\Gamma_n(c)}{\Gamma_n(a)\Gamma_n(c-a)}
\int_0^{1_n} |s|^{a-p} |1-s|^{c-a-b-p} \mathrm{d}s =
\frac{\Gamma_n(c)}{\Gamma_n(a)\Gamma_n(c-a)}\cdot
\frac{\Gamma_n(a)\Gamma_n(c-a-b)}{\Gamma_n(c-b)}
$$
by \eqref{eq_3.11}. Hence
\begin{flalign} 
\mbox{(Gauss)} \quad
_2F_1(a,b;c;1_n)=
\frac{\Gamma_n(c)\Gamma_n(c-a-b)}
{\Gamma_n(c-a)\Gamma_n(c-b)}. &&
\label{eq_4.1} 
\end{flalign}

Next, from \eqref{eq_3.17} in the form
$$
_2F_1(c-a,c-b;c;t)=|1-t|^{a+b-c} \mbox{}_2F_1(a,b;c;t)
$$
we obtain
$$
\sum_\lambda \frac{(c-a)_\lambda (c-b)_\lambda}{(c)_\lambda} J_\lambda^*
=\sum_{\mu,\nu} 
\frac{(a)_\mu (b)_\mu}{(c)_\mu} (c-a-b)_\nu J_\mu^* J_\nu^*
$$
and hence
\begin{equation} 
 \frac{(c-a)_\lambda (c-b)_\lambda}{(c)_\lambda}
=\sum_{\mu,\nu} 
\frac{(a)_\mu (b)_\mu}{(c)_\mu} (c-a-b)_\nu \langle J_{\lambda/\mu}, 
J_\nu^* 
\rangle
\label{eq_4.2} 
\end{equation}
since $\langle J_\lambda, J_\mu^*J_\nu^* \rangle = \langle 
J_{\lambda/\mu}, J_\nu^* 
\rangle$.

Suppose in particular that $\lambda = (N^n), \, N\ge 0$. For a partition 
$\mu\subset (N^n)$, define $\hat \mu$ to be the complement of $\mu$ in 
$(N^n)$, i.e., 
\begin{equation}
\hat \mu_i=N-\mu_{n+1-i}. 
\label{eq_4.3} 
\end{equation}

We have then
\begin{equation} 
(a)_{\hat \mu} (-a-N+p)_\mu = (-1)^{|\mu|} (a)_{(N^n)}.
\label{eq_4.4} 
\end{equation}
\begin{proof}
\begin{align*}
(a)_{\hat \mu} 
& = \prod_{i=1}^n (a-\tfrac{1}{2}(i-1))_{\hat \mu_i} \\
& = \prod_{j=1}^n (a-\tfrac{1}{2}(n-j))_{N-\mu_j} \\
& = (a)_{(N^n)}\prod_{j=1}^n 
\prod_{i=1}^{\mu_j}(a-\tfrac{1}{2}(n-j)+N-i)^{-1} \\
& = (-1)^{|\mu|}(a)_{(N^n)}\prod_{j=1}^n 
\prod_{i=1}^{\mu_j}(-a-N+\tfrac{1}{2}(n-j)+i)^{-1} \\
& = (-1)^{|\mu|}(a)_{(N^n)}/ (-a-N+p)_\mu.
\end{align*}
\end{proof}

\begin{flalign}
\mbox{We have }P_{\hat\mu}(x) = |x|^N P_\mu(x^{-1}). &&
\label{eq_4.5} 
\end{flalign}
\begin{proof}
Both sides have $x^{\hat \mu}$ as leading term, and
$$
\left\langle 
|x|^N P_\mu(x^{-1}),
|x|^N P_\nu(x^{-1})
 \right\rangle_2' = 
\left\langle P_\mu,P_\nu\right\rangle_2' ,
$$
which is 0 if $\mu\ne\nu$.
\end{proof}

Let $f_{\nu\hat\mu}^{(N^n)}$ denote the coefficient of $P_{(N^n)}$ in 
$P_\nu P_{\hat\mu}$. Then
\begin{flalign} 
\mbox{(i)} \quad& f_{\nu\hat\mu}^{(N^n)} = 0, \mbox{ if }\nu\ne \mu;&& 
\label{eq_4.6}
\\
\mbox{(ii)} \quad & f_{\mu\hat\mu}^{(N^n)} = |P_\mu|'^2 / |1|'^2 = 
P_\mu(1_n)/Q_{\hat\mu}(1_{n+1}). &&
\nonumber
\end{flalign}
\begin{proof}
We have
$$
\left\langle P_\mu,P_\nu\right\rangle' = 
\left\langle P_\nu\bar P_\mu,1\right\rangle' = 
\left\langle P_\nu P_{\hat\mu},|x|^N\right\rangle' = 
\sum_\lambda f_{\nu\hat\mu}^\lambda \langle
P_\lambda,|x|^N\rangle' = f_{\nu\hat\mu}^{(N^n)} 
\langle 1,1 \rangle'
$$
by orthogonality, since $|x|^N = P_{(N^n)}$ (in $n$ variables).

This proves (i) and (ii) since (for zonal polynomials) the conjecture
$$
|P_\mu|'^2 = |1|'^2 P_\mu(1_n)/Q_\mu(1_{n+1})
$$
is known to be true, and $Q_\mu(1_{n+1})=Q_{\hat \mu}(1_{n+1})$.
\end{proof}

We now return to \eqref{eq_4.2}, with $\lambda=(N^n)$:
$$
 \frac{(c-a)_{(N^n)} (c-b)_{(N^n)}}{(c)_{(N^n)}}
=\sum_{\mu,\nu} 
\frac{(a)_\mu (b)_\mu}{(c)_\mu} (c-a-b)_\nu
 \langle J_{(N^n)/\mu}, 
J_\nu^* 
\rangle.
$$

Now
\begin{align*}
 J_{(N^n)/\mu} &=h'_{(N^n)/\mu} Q_{(N^n)/\mu}
\\
&=h'_{(N^n)/\mu} \sum_\nu f_{\mu\nu}^{(N^n)}
Q_\nu
\\
&=h'_{(N^n)/\mu} \frac{P_\mu(1_n)}{Q_{\hat\mu}(1_{n+1})}Q_{\hat\mu}
\quad\quad\mbox{by \eqref{eq_4.6}},
\end{align*}
hence $\langle J_{(N^n)/\mu}, J_\nu^* \rangle=0$ if $\nu\ne\hat \mu$ and
is equal to 
$$
\frac{h'_{(N^n)}}{h_\mu' h_{\hat\mu}'}\cdot 
\frac{P_\mu(1_n)}{Q_{\hat\mu}(1_{n+1})}
$$
if $\nu=\hat\mu$. So we obtain
$$
 \frac{(c-a)_{(N^n)} (c-b)_{(N^n)}}{(c)_{(N^n)}}
=\sum_{\mu\subset (N^n)} 
\frac{(a)_\mu (b)_\mu}{(c)_\mu} 
(c-a-b)_{\hat\mu}
\frac{h'_{(N^n)}}{h_\mu' h_{\hat\mu}'} \cdot
\frac{P_\mu(1_n)}{Q_{\hat\mu}(1_{n+1})}.
$$
By \eqref{eq_4.4} we have
$$
(c-a-b)_{\hat\mu} = \frac{(-1)^{|\hat\mu|}(c-a-b)_{(N^n)}}
{(a+b-c-N+p)_\mu}
$$
and
$$
(-N)_\mu = (-1)^{|\mu|} (p)_{(N^n)}/(p)_{\hat\mu}.
$$
Hence
$$
(c-a-b)_{\hat\mu}
\frac{h'_{(N^n)}}{h_\mu' h_{\hat\mu}'} \cdot
\frac{P_\mu(1_n)}{Q_{\hat\mu}(1_{n+1})}
=
\frac{(-N)_\mu}{(a+b-c-N+p)_\mu}
(c-a-b)_{(N^n)}
\frac{h'_{(N^n)}}{(p)_{(N^n)}}\cdot
\frac{(p)_{\hat\mu}}{J_{\hat\mu}(1_{n+1})}
J_\mu^*(1_n).
$$

But $h'_{(N^n)}=2^{Nn} (p)_{(N^n)}$, and $J_{\hat\mu}(1_{n+1})=
2^{|\hat\mu|}(p)_{\hat\mu}$, so that we finally obtain
$$
\frac{(-N)_\mu}{(a+b-c-N+p)_\mu}
(c-a-b)_{(N^n)}
2^{|\mu|}
J_\mu^*(1_n)
$$
and therefore
\begin{eqnarray*}
\frac
{(c-a)_{(N^n)}(c-b)_{(N^n)}}
{(c)_{(N^n)}(c-a-b)_{(N^n)}}
&=& \sum_\mu
\frac{(a)_\mu (b)_\mu (-N)_\mu}{(c)_\mu (d)_\mu} 2^{|\mu|}J_\mu^*(1_n) \\
&=& _3F_2(a,b,-N;c,d;1),
\end{eqnarray*}
where $d=a+b-c-N+p$, i.e.,
$$
a+b-N+p=c+d.
$$

So, changing the notation, we have established the analogue of Saalschutz's 
summation
\begin{flalign}
\mbox{Let }& \underline{a}=(a_1,a_2,a_3), \, \underline{b}=(b_1,b_2), 
\mbox{ where}&&
\label{eq4.7}\\
\mbox{(i)} &\quad \sum a_i+p = \sum b_i \quad (\mbox{or say 
}|\underline{a}|+p=|\underline{b}|)
\nonumber 
\\
\mbox{(ii)}& \quad\mbox{one of the $a_i$ is a negative integer}.
\nonumber
\end{flalign}

Then 
$$
_3F_2(\underline{a};\underline{b};1_n)=\prod_I 
\Gamma_n(b_1-a_{I})^{(-1)^{|I|+1}},
$$
where the product is taken over the 8 subsets $I$ of $\{1,2,3\}$, and 
$a_I=\sum_{i\in I}a_i$.
\newpage

\section{Integral formulae II}
\label{sec_5}

Let $X$ (resp.\ $X^+$) denote the set of diagonal matrices 
$x={\rm diag}(x_1,x_2,\ldots,x_n)$ with all $x_i\ge 0$ (resp.\ with $x_1\ge 
x_2\ge \cdots \ge x_n\ge 0$).

Each $s\in\Sigma^+$ is of the form $kxk'$ ($k\in K, x\in X^+$): not 
uniquely, since we may replace $k$ by $k\varepsilon$, where $\varepsilon$ 
is a diagonal matrix of $\pm 1$'s (so that here are $2^n$ possibilities 
for $\varepsilon$). Corresponding to the decomposition $s=kxk'$ we have the 
integral formula (cf.\ Farrell \cite{Farrell85}, p.\ 75: recall that 
$\mathrm{d}s=\pi^{-n(n-1)/4} \prod_{i\le j} \mathrm{d}s_{ij}$)

\begin{equation} 
\int_{\Sigma^+} f(s)\mathrm{d}s = \frac{1}{c_n'}\int_X\left( 
\int_K f(kxk')\mathrm{d}k
\right) |\Delta(x)|\mathrm{d}x
\label{eq_5.1} 
\end{equation}
where $\mathrm{d}x=\prod_{i=1}^n \mathrm{d}x_i$, and

\begin{equation} 
\Delta(x)=\prod_{i<j}(x_i-x_j)
\label{eq_5.2} 
\end{equation}
and
\begin{equation} 
c_n'=\frac{n!}{\pi^{n/2}}\Gamma_n(\tfrac{n}{2})=\prod_{i=1}^n
\frac{(i/2)!}{(1/2)!}.
\label{eq_5.3} 
\end{equation}
In particular, when $f$ is $K$-invariant, \eqref{eq_5.1} takes the simpler 
form
\begin{equation}
\int_{\Sigma^+} f(s)\mathrm{d}s = \frac{1}{c_n'}\int_X
 f(x)
 |\Delta(x)|\mathrm{d}x.
\label{eq_5.4} 
\end{equation}

From \eqref{eq_3.3} we obtain, since
$$
\int_K e^{-\tr(kxk'y)}\mathrm{d}k = \mbox{}_0F_0(-x;y)
$$
\begin{equation}
\int_X \mbox{}_0F_0(-x;y)|x|^{a-p}|\Delta(x)| \mathrm{d}x = c_n'|y|^{-a} 
\Gamma_n(a) 
\label{eq_5.5} 
\end{equation}
(but be warned that this integral cannot be evaluated term by term).

In particular, when $y=1_n$ we obtain
\begin{equation}
\int_X e^{-\tr(x)}|x|^{a-p}|\Delta(x)| \mathrm{d}x = c_n' 
\Gamma_n(a).
\label{eq_5.6} 
\end{equation}

Next \eqref{eq_3.6} gives in the same way
\begin{equation}
\int_X e^{-\tr(x)}\Omega_\lambda(x)|x|^{a-p}|\Delta(x)| \mathrm{d}x = c_n' 
\Gamma_n(a;\lambda),
\label{eq_5.7} 
\end{equation}
which reduces to \eqref{eq_5.6} when $\lambda=0$. Likewise, from 
\eqref{eq_3.7},
\begin{equation}
\int_X \mbox{}_0F_0(-x;y)\Omega_\lambda(x)|x|^{a-p}|\Delta(x)| \mathrm{d}x = 
c_n'|y|^{-a} \Gamma_n(a;\lambda) \Omega_\lambda(y^{-1}). 
\label{eq_5.8} 
\end{equation}
(Again this integral cannot be evaluated term by term.)

Next, if $f$ is a $K$-invariant function on $\Sigma^+$, its Laplace 
transform $Lf$ is also $K$-invariant, since
\begin{align*}
(Lf)(t)
&=\int_{\Sigma^+} e^{-\tr(st)}f(s) \mathrm{d}s\\
&=\int_{\Sigma^+} e^{-\tr(ksk't)}f(s) \mathrm{d}s\\
&=(Lf)(k'tk).
\end{align*}

We have
\begin{equation}
(Lf)(y) = c_n'^{-1} \int_X \mbox{}_0F_0(-x;y) f(x)|\Delta(x)|\mathrm{d}x. 
\label{eq_5.9} 
\end{equation}

Next, \eqref{eq_3.10} gives
\begin{equation}
\left(\int_0^1\right)^n\Omega_\lambda(x) |x|^{a-p} |1-x|^{b-p}
|\Delta(x)|\mathrm{d}x =
c_n'\frac{\Gamma_n(a;\lambda)\Gamma_n(b)}{\Gamma_n(a+b;\lambda)}, 
\label{eq_5.10} 
\end{equation}
which is Kadell's generalization of Selberg's integral with parameter 
$k=\frac{1}{2}$.

Finally, from \eqref{eq_3.13} and \eqref{eq_3.14}, bearing in mind 
\eqref{eq_1.8}, we obtain
\begin{align}
\int_X e^{-\tr(x)} F(\underline{a};\underline{b};x;y)|x|^{a-p}|\Delta(x)| 
\mathrm{d}x 
&= 
c_n' \Gamma_n(a) F(\underline{a}^+;\underline{b};y), 
\label{eq_5.11} 
\\
\left(\int_0^1\right)^n
F(\underline{a};\underline{b};x;y)|x|^{a-p} |1-x|^{b-a-p}|\Delta(x)| 
\mathrm{d}x 
&= 
c_n' \frac{\Gamma_n(a)\Gamma_n(b-a)}{\Gamma_n(b)} 
F(\underline{a}^+;\underline{b}^+;y), 
\label{eq_5.12} 
\end{align}
with the notation of \S3 for $\underline{a}, 
\underline{a}^+, \underline{b}, 
\underline{b}^+$.
\newpage

\section{Hypergeometric functions with parameter $\alpha$}

Since the zonal polynomials $Z_\lambda$ are Jack's polynomials 
$J_\lambda(x;\alpha)$ with parameter $\alpha=2$, it is clear how to 
generalize the definitions. 

Throughout, we shall set
\begin{equation}
k=\alpha^{-1}. 
\label{eq_6.1} 
\end{equation}

Define
\begin{equation}
(a)_\lambda = (a;\alpha)_\lambda = \prod_{i\ge 1}(a-k(i-1))_{\lambda_i}
\label{eq_6.2} 
\end{equation}
and then
\begin{equation}
(\underline{a})_\lambda = (\underline{a};\alpha)_\lambda
=(a_1)_\lambda \cdots (a_p)_\lambda
\label{eq_6.3} 
\end{equation}
for $\underline{a}=(a_1,\ldots,a_p)$.

With this notation established, we define
\begin{align}
_pF_q(\underline{a};\underline{b};x;\alpha)
&=\sum_\lambda\frac{(\underline{a})_\lambda}{(\underline{b})_\lambda}
\alpha^{|\lambda|}J_\lambda^*(x;\alpha)
\label{eq_6.4}
\\
_pF_q(\underline{a};\underline{b};x,y;\alpha)
&=\sum_\lambda\frac{(\underline{a})_\lambda}{(\underline{b})_\lambda}
\alpha^{|\lambda|}
\frac{J_\lambda^*(x;\alpha)J_\lambda^*(y;\alpha)}
     {J_\lambda^*(1_n;\alpha)},
\label{eq_6.5}
\end{align}
where $p$, $q$ are the lenghts of the sequences $\underline{a}$, 
$\underline{b}$, respectively.

Notice that
\begin{equation} 
_pF_q(\underline{a};\underline{b};x,1;\alpha)=
\mbox{}_pF_q(\underline{a};\underline{b};x;\alpha).
\label{eq_6.6} 
\end{equation}

As before (\S2) we have
\begin{equation}
_0F_0(x;\alpha)=e^{p_1(x)}.
\label{eq_6.7} 
\end{equation}
For the same reason as before: the coefficient of $p_1^m$ in 
$J_\lambda(x;\alpha)$ is 1 for all $\lambda\vdash m$, so that
$\langle p_1^m,J_\lambda \rangle_\alpha = \alpha^m m!$ and therefore
$$
p_1^m = \sum_{\lambda\vdash m} 
\langle p_1^m,J_\lambda \rangle J_\lambda^* = \alpha^m m!
\sum_{\lambda\vdash m} J_\lambda^*.
$$

\begin{equation}
_1F_0(a;x;\alpha) = |1-x|^{-a}. 
\label{eq_6.8} 
\end{equation}
The proof is the same as that of \eqref{eq_2.2}, using
\begin{equation}
\varepsilon_X(J_\lambda)=\alpha^{|\lambda|} (kX)_\lambda  
\label{eq_6.9} 
\end{equation}
in place of \eqref{eq_2.3}. We find then that
$$
_1F_0(a;x;\alpha)=\varepsilon_{a\alpha}^{(y)} 
\prod_{i,j=1}^n(1-x_iy_j)^{-1/\alpha}
=\prod_{i=1}^n(1-x_i)^{-a}.
$$

The formula \eqref{eq_2.4} has no counterpart for general $\alpha$ since 
there is no $K$ over which to integrate.

\begin{equation}
_1F_0(kn;x,y;\alpha) = \Pi(x,y;\alpha) = \prod_{i,j}(1-x_iy_j)^{-k}. 
\label{eq_6.10} 
\end{equation}

\begin{proof}
We have
\begin{align*}
_1F_0 
&= \sum_\lambda \alpha^{|\lambda|}(kn)_\lambda
\frac{J_\lambda(x)}{J_\lambda(1_n)}J_\lambda^*(y)
\\
&= \sum_\lambda J_\lambda(x)J_\lambda^*(y) \quad \mbox{by \eqref{eq_6.9}}
\\
&=\Pi(x,y;\alpha).
\end{align*}
\end{proof}

Let $\Pi''(x,y;\alpha)$ be the defining series for the scalar product
$$
\langle f,g \rangle_\alpha'' = 
\langle f,g \rangle_\alpha' /
\langle 1,1 \rangle_\alpha'.
$$

Conjecturally,
\begin{equation} 
\langle P_\lambda,P_\lambda\rangle_\alpha''
=\frac{\varepsilon_n(P_\lambda)}{\varepsilon_{n-1+\alpha}(Q_\lambda)},
\tag{6.11C}
\label{eq_6.11C} 
\end{equation}
from which it would follow that
\begin{align*}
\Pi''(x,y;\alpha) 
&=\sum_\lambda \frac{P_\lambda(x)P_\lambda(y)}{\varepsilon_n(P_\lambda)}
\varepsilon_{n-1+\alpha}(Q_\lambda)
\\
&=\sum_\lambda 
\Omega_\lambda(x) 
J^*_\lambda(y)
\varepsilon_{n-1+\alpha}(J_\lambda)
\\
&=\sum_\lambda 
(k(n-1)+1)_\lambda \alpha^{|\lambda|}
\Omega_\lambda(x) 
J^*_\lambda(y),
\end{align*}
i.e., that
\begin{equation} 
\Pi''(x,y;\alpha) = \mbox{} _1F_0(p;x,y;\alpha),
\tag{6.12C}
\label{eq_6.12C} 
\end{equation}
where
\setcounter{equation}{12}
\begin{equation} 
\label{eq_6.13} 
p=k(n-1)+1.
\end{equation}

\underline{Remark}. When $n=1$ these hypergeometric series agree with the 
classical ones, and do not depend on $\alpha$. For by \ref{eq_6.7} we 
have
$e^x=\sum_{r\ge 0} \alpha^r J_{(r)}^*(x;\alpha)$, giving
$\alpha^r J_{(r)}^* = x^r/r!$ for each $r\ge 0$.

\underline{Question}. In the case $\alpha=2$ (also $\alpha=1,\frac{1}{2}$ 
probably) we have \eqref{eq_2.7}
$$
_0F_0(1+x,y) = \int_K e^{{\rm trace}((1+x)kyk')}\mathrm{d}k =
e^{{\rm trace}(y)} \mbox{}_0F_0(x,y),
$$
from which it follows that
$$
\Omega_\lambda(1+x) = \sum_\mu \Omega_\mu(x)
\langle J_{\lambda/\mu}, e^{p_1} \rangle_\alpha \cdot \alpha 
^{-|\lambda-\mu|}. 
$$
In other words, if we write
\begin{align}
\binom{\lambda}{\mu} 
&= \mbox{coefficient of $p_1^m$ in $J_{\lambda/
\mu} \quad\quad(\lambda\supset\mu, |\lambda-\mu|=m)$} 
\label{eq_6.14} 
\\
&=\varepsilon(J_{\lambda/\mu}),\;\mbox{ where 
$\varepsilon(p_r)=\delta_{1r}$}\quad\quad(\mbox{so 
}\varepsilon(J_\lambda)=1),
\nonumber
\end{align}
then
\begin{equation} 
\Omega_\lambda(1+x)=\sum_{\mu\subset\lambda}\binom{\lambda}{\mu}\Omega_\mu(x).
\label{eq_6.15} 
\end{equation}
Is this true for all values of $\alpha$? It is a sort of substitute for the doubling principle.

We come now to the analogues of the integral formulas of \S\ref{sec_5}.

Let 
\begin{equation}
c_n'(\alpha)=\prod_{i=1}^n \frac{(ik)!}{k!} = n! \prod_{i=1}^n \frac{\Gamma(ik)}{\Gamma(k)} 
\label{eq_6.16} 
\end{equation}
and define
\begin{align*}
\Gamma_n(a;\alpha)&=\prod_{i=1}^n \Gamma(a-k(i-1))
\\
\Gamma_n(a;\lambda;\alpha)&=\prod_{i=1}^n \Gamma(a+\lambda_i-k(i-1))
\end{align*}
($\lambda$ a partition of length $\leq n$).

Then we have, in generalization of \eqref{eq_5.6}, \eqref{eq_5.7}, \eqref{eq_5.10}:--
\begin{align}
\int_{\mathbb{R}^n_+} e^{-\tr(x)}|x|^{a-p} |\Delta(x)|^{2k} \mathrm{d}x &= c_n'(\alpha) \Gamma_n(a;\alpha),
\label{eq_6.17}
\\
\int_{\mathbb{R}^n_+} e^{-\tr(x)}\Omega_\lambda(x;\alpha)|x|^{a-p} |\Delta(x)|^{2k} \mathrm{d}x &= c_n'(\alpha) \Gamma_n(a;\lambda;\alpha),
\label{eq_6.18}
\\
\int_{[0,1]^n} \Omega_\lambda(x;\alpha)|x|^{a-p}|1-x|^{b-p} |\Delta(x)|^{2k} \mathrm{d}x &= c_n'(\alpha) \frac{
\Gamma_n(a;\lambda;\alpha) \Gamma_n(b;\alpha)}{\Gamma_n(a+b;\lambda;\alpha)}.
\label{eq_6.19}
\end{align}
\eqref{eq_6.19} is Kadell's extension of Selberg's integral. To deduce \eqref{eq_6.18} from \eqref{eq_6.19}, put $x_i=y_i/N$, where $N=b-p$. We obtain 
$$\left(\int_0^N \right)^n \Omega_\lambda(y;\alpha) |y|^{a-p} \left| 1-\tfrac{y}{N}\right|^N
 |\Delta(y)|^{2k} \mathrm{d}y=N^c \cdot c_n'(\alpha) \Gamma_n(a;\lambda;\alpha)
 \prod_{i=1}^n \frac{\Gamma(N+p-k(i-1))}{\Gamma(N+a+\lambda_i+p-k(i-1))},
$$
where
$$
c=|\lambda|+n(a-p)+kn(n-1)+n = |\lambda|+na = \sum_{i=1}^n (a+\lambda_i).
$$
Now let $N\rightarrow\infty$; then
$$
\left| 1-\tfrac{y}{N}\right|^N = \prod_{i=1}^n \left(1-\tfrac{y_i}{N} \right)^N\rightarrow \prod e^{-y_i}
=e^{-\tr(y)}
$$
and
$$
\frac{\Gamma(N+u)}{\Gamma(N+v)}\sim N^{u-v} \quad\mbox{as }N\rightarrow\infty
$$
(any $u,v$). In the limit, therefore, we get \eqref{eq_6.18}. Finally \eqref{eq_6.17} is the case $\lambda=0$ of \eqref{eq_6.18}.

The integral formulas \eqref{eq_6.18}, \eqref{eq_6.19} lead directly to the following formulas for the hypergeometric functions:\footnote{notation: d$\mu(x) = c_n'(\alpha)^{-1}|\Delta(x)|^{2k}$d$x$}
\begin{equation}
\int_{\mathbb{R}^n_+}e^{-\tr(x)} \mbox{} _pF_q(\underline{a};\underline{b};x,y;\alpha)
|x|^{a-p} |\Delta(x)|^{2k} \mathrm{d}x 
=
c_n'(\alpha) \Gamma_n(a;\alpha) \,_{p+1}F_q(\underline{a}^+;\underline{b};y;\alpha),
\label{eq_6.20}
\end{equation}
\begin{multline}
\left(\int_0^1 \right)^n
\mbox{}
_pF_q(\underline{a};\underline{b};x,y;\alpha)
|x|^{a-p} |1-x|^{b-a-p}|\Delta(x)|^{2k} \mathrm{d}x 
\\
=
c_n'(\alpha) 
\frac{\Gamma_n(a;\alpha)\Gamma_n(b-a;\alpha)}{\Gamma_n(b;\alpha)}
 \,_{p+1}F_{q+1}(\underline{a}^+;\underline{b}^+;y;\alpha),
\label{eq_6.21}
\end{multline}
where as before
\begin{alignat*}{2}
\underline{a} & =   (a_1,\ldots,a_p), &\qquad
\underline{a}^+  & =   (a_1,\ldots,a_p,a),
\\
\underline{b} & =  (b_1,\ldots,b_q), & \qquad
\underline{b}^+ & = (b_1,\ldots,b_q,b).
\end{alignat*}
\bigskip

\underline{Question}: We may ask whether \eqref{eq_2.8} holds for all $\alpha$, i.e., whether
$$
_1F_0(a;1+x,y;\alpha) =|1-y|^{-a} \mbox{} _1F_0(a;x,\tfrac{y}{1-y};\alpha).
$$
Spelt out, this is
$$
\sum_\lambda (a)_\lambda \alpha^{|\lambda|} J_\lambda^*(1+x;\alpha)
\Omega_\lambda(y;\alpha)
=|1-y|^{-a}\sum_\mu(a)_\mu \alpha^{|\mu|} J_\mu^*(x;\alpha) \Omega_\mu(\tfrac{y}{1-y};\alpha).
$$

If we equate the coefficients of $\Omega_\lambda(y;\alpha)$ on either side we obtain a (purported) polynomial identity in $a$. So there is no loss of generality in taking $a=-N$, where $N$ is an integer $\ge \lambda_1$, so that $\lambda\subset (N^n)$. Then the assertion is 
\begin{flalign} 
(-N)_\lambda \alpha^{|\lambda|}J_\lambda^*(1+x;\alpha)
\label{eq_6.22} 
=\mbox{coefficient of }\Omega_\lambda(y;\alpha)\mbox{ in }&&
\end{flalign}
$$
|1-y|^N \sum_\mu
(-N)_\mu \alpha^{|\mu|} J_\mu^*(x;\alpha) \Omega_\mu(\tfrac{y}{1-y};\alpha).
$$

Now we have (see \eqref{eq_4.5})
$$
P_{\hat\mu}(x;\alpha)=|x|^NP_\mu(x^{-1};\alpha)
$$
and hence
$$
P_{\hat\mu}(1_n)=P_\mu(1_n)
$$
so that
\begin{equation} 
\Omega_{\hat\mu}(x;\alpha)=|x|^N\Omega_\mu(x^{-1};\alpha).
\label{eq_6.23} 
\end{equation}

In what follows we shall assume (a) that \eqref{eq_6.15} is true, i.e., 
\begin{equation}
\Omega_\lambda(1+x;\alpha)=
\sum_{\mu\subset\lambda} \binom{\lambda}{\mu} \Omega_\mu(x;\alpha), 
\label{eq_6.24} 
\end{equation}
where $\binom{\lambda}{\mu}$ is the coefficient of $p_1^{|\lambda-\mu|}$ in $J_{\lambda/\mu}$, and that
\begin{equation}
\langle P_\lambda,P_\lambda \rangle''  = \varepsilon_n(P_\lambda)/\varepsilon_{n-1+\alpha}(Q_\lambda),
\label{eq_6.25} 
\end{equation}
where $\langle f,g \rangle''=\langle f,g \rangle'/\langle 1,1 \rangle'$.

Then we have by \eqref{eq_6.23}
\begin{align*}
|1-y|^N \Omega_\mu(\tfrac{y}{1-y};\alpha) 
&=|y|^N \Omega_{\hat\mu}(y^{-1}-1;\alpha)
\\
&=(-1)^{|\hat\mu|}|y|^N \Omega_{\hat\mu}(1-y^{-1};\alpha)
\\
&=(-1)^{|\hat\mu|}|y|^N \sum_{\nu\supset\mu}(-1)^{|\hat\nu|}\binom{\hat\mu}{\hat\nu}\Omega_{\nu}(y^{-1};\alpha)
\\
&=\sum_{\nu\supset\mu}(-1)^{|\nu-\mu|}\binom{\hat\mu}{\hat\nu}\Omega_{\nu}(y;\alpha)
\end{align*}
by \eqref{eq_6.24} and \eqref{eq_6.23} again.

Hence we obtain from \eqref{eq_6.22}
$$
(-N)_\lambda \alpha^{|\lambda|} J_\lambda^*(1+x;\alpha)
=\sum_{\mu\subset\lambda} (-1)^{|\lambda-\mu|} (-N)_\mu \alpha^{|\mu|}
\binom{\hat\mu}{\hat\lambda} J_\mu^*(x;\alpha),
$$
which we want to compare with \eqref{eq_6.24}. First of all, by \eqref{eq_4.4}
(valid for all $\alpha$) we have
$$
(-N)_\lambda = (-1)^{|\lambda|} (p)_{(N^n)}/(p)_{\hat\lambda}
$$
and likewise with $\mu$ in place of $\lambda$; also\footnote{See Additional observation 1 at the end of this section.}
$$
\alpha^{|\hat\lambda|} (p)_{\hat\lambda} = \varepsilon_{n-1+\alpha}(J_{\hat\lambda}).
$$

So we obtain
\begin{equation}
\frac{J_\lambda^*(1+x;\alpha)}{\varepsilon_{n-1+\alpha}(J_{\hat\lambda})} = 
\sum_{\mu\subset\lambda} \binom{\hat\mu}{\hat\lambda} 
\frac{J_\mu^*(x;\alpha)}{\varepsilon_{n-1+\alpha}(J_{\hat\mu})}.
\label{eq_6.26} 
\end{equation}
Comparison of this with \eqref{eq_6.24} requires us to establish
\begin{equation}
\frac{\varepsilon_n(J_{\lambda}^*)}{\varepsilon_{n-1+\alpha}(J_{\hat\lambda})} 
\binom{\lambda}{\mu}
=
\frac{\varepsilon_n(J_{\mu}^*)}{\varepsilon_{n-1+\alpha}(J_{\hat\mu})}.  
\binom{\hat\mu}{\hat\lambda} 
\label{eq_6.27} 
\end{equation}

To prove this we proceed as follows.
\begin{flalign}
\mbox{If }\mu\subset\lambda\mbox{ then }&&
\label{eq_6.28} 
\end{flalign}
$$Q_{\lambda/\mu} \langle P_\lambda,P_\lambda \rangle'
=Q_{\hat\mu/\hat\lambda} \langle P_\mu,P_\mu \rangle'.$$
\begin{proof}
Let $f_{\mu\nu}^\lambda$ as usual denote the coefficient of $P_\lambda$ in 
$P_\mu P_\nu$. Then by orthogonality we have
\begin{align*}
f_{\mu\nu}^\lambda 
&=
\frac
{\langle P_\lambda, P_\mu P_\nu\rangle'}
{\langle P_\lambda, P_\lambda\rangle'}
\\
&=
\frac
{\langle P_{\hat\mu}, P_{\hat\lambda}P_\nu\rangle'}
{\langle P_\lambda, P_\lambda\rangle'}
\end{align*}
by \eqref{eq_4.5}. Hence
\begin{align*}
\langle P_\lambda, P_\lambda\rangle'
f_{\mu\nu}^\lambda 
&=
\langle P_{\hat\mu}, P_{\hat\mu}\rangle'
f_{\hat\lambda\nu}^{\hat\mu} 
\\
&=
\langle P_{\mu}, P_{\mu}\rangle'
f_{\hat\lambda\nu}^{\hat\mu}.
\end{align*}

Since $Q_{\lambda/\mu} = \sum_\nu f_{\mu\nu}^\lambda Q_\nu$, the result follows.\footnote{See Additional observation 2 at the end of this section.}
\end{proof}

From \eqref{eq_6.28} and \eqref{eq_6.25} we have
\begin{align*}
\frac{\varepsilon_n(J_\lambda^*)}{\varepsilon_{n-1+\alpha}(J_{\hat\lambda})}
J_{\lambda/\mu}
&=
\frac{\langle P_\lambda,P_\lambda \rangle ''}{h_{\hat\lambda}' h_\mu'} Q_{\lambda/\mu}
\\
&=
\frac{\langle P_{\hat\mu},P_{\hat\mu} \rangle ''}{h_{\hat\lambda}' h_\mu'} Q_{\hat\mu/\hat\lambda}
\\
&=
\frac{\varepsilon_n(J_\mu^*)}{\varepsilon_{n+1-\alpha}(J_{\hat\mu})}
J_{\hat\mu/\hat\lambda}
\end{align*}
from which \eqref{eq_6.27} follows. Hence
\begin{flalign}
\mbox{If \eqref{eq_6.24} and \eqref{eq_6.25} hold, then} &&
\label{eq_6.29} 
\end{flalign}
$$
_1F_0(a;1+x,y;\alpha) =|1-y|^{-a} \mbox{} _1F_0(a;x,\tfrac{y}{1-y};\alpha)
$$
for all $a,x,y,\alpha$. $\Box$
\bigskip

\noindent\underline{Remark}: The conjecture \eqref{eq_6.25} (from Ch.\ VI \S9) can be reformulated as follows:
\begin{equation} 
\langle |1-x|^{-p}, \mbox{} _1F_0(nk;x,y;\alpha) \rangle'' = |1-y|^{-nk}.
\label{eq_6.25p}
\tag{6.25$'$} 
\end{equation}

For 
\begin{align*}
|1-x|^{-p} 
&=
\sum_\lambda \alpha^{|\lambda|} (p)_\lambda J_\lambda^*(x) 
\\
&=
\sum_\lambda \varepsilon_{n-1+\alpha} (J_\lambda) J_\lambda^*(x) 
\\
&=
\sum_\lambda \varepsilon_{n-1+\alpha} (Q_\lambda) P_\lambda(x) 
\end{align*}
and
\begin{align*}
_1F_0(nk;x,y;\alpha) &= \sum_\mu \alpha^{|\mu|} (nk)_\mu
\frac{J_\lambda(x)}{J_\lambda(1_n)} J_\lambda^* (y)
\\
&=
\sum_\mu P_\lambda(x) Q_\lambda(y)
\end{align*}
so that the left hand side of \eqref{eq_6.25p} is equal to
\begin{align*}
\sum_\lambda\varepsilon_{n-1+\alpha}(Q_\lambda)\langle P_\lambda,P_\lambda\rangle''
Q_\lambda(y) &=
\sum_\lambda \varepsilon_n(P_\lambda) Q_\lambda(y) \quad \mbox{ by \eqref{eq_6.25}}\\
&=|1-y|^{-nk}. \quad\Box
\end{align*}

Note that 
$$
_1F_0(nk;x,y;\alpha) = \prod_{i,j} (1-x_iy_j)^{-k} = |1-x\otimes y|^{-k}.
$$
If we write
$$
f(t)=\prod \left(\frac{1-y_i}{1-ty_i} \right)^k,
$$
then \eqref{eq_6.25p} takes the form
$$
\langle |1-x|^{-p},f(x_1)\cdots f(x_n)\rangle'' = 1.
$$

\subsection*{Additional observation 1}

The "formal degree" is
\begin{eqnarray*}
d_{\lambda}(\alpha) &=& \langle\Omega_{\lambda},\Omega_{\lambda}\rangle''^{\,-1} \\
&=& \varepsilon_n(P_{\lambda})\varepsilon_{n-1+\alpha}(Q_{\lambda}) \\
&=& \prod_{i<j}\frac{(\xi_i-\xi_j)_k(\xi_i-\xi_j+1-k)_k}{(k(j-1))_k(k(j-1)+1-k)_k}
\end{eqnarray*}
where $\xi=\lambda+k\delta$,   $\delta = (n-1,\ldots,1,0)$

So 
$$d_{\lambda}(\alpha) = \prod_{i\neq j}\frac{(\xi_i-\xi_j)_k}{(k\delta_i-k\delta_j)_k}.$$
This is another version of \eqref{eq_6.25}.
\bigskip

\subsection*{Additional observation 2}

From \eqref{eq_6.28} we have
$$\frac{\langle P_{\lambda},P_{\lambda}\rangle'}{\langle P_{\mu},P_{\mu}\rangle'} = \frac{Q_{\hat{\mu}/\hat{\lambda}}}{Q_{\lambda/\mu}} =  \frac{\phi_{\hat{\mu}/\hat{\lambda}}}{\phi_{\lambda/\mu}}$$
if $|\lambda - \mu| = 1$.

This \underline{must} lead to another proof of the conjecture \eqref{eq_6.25}, by induction on $|\lambda|$.
\newpage

\section*{Duality}

The formula is
$$
\omega_\alpha \,\mbox{} _pF_q(\underline{a}; \underline{b};x;\alpha)=
\mbox{}_pF_q(-\alpha\underline{a};-\alpha\underline{b}; (-1)^{p-q} \alpha^{q-p+1}x;\alpha^{-1}).
$$
\begin{proof}
We have
$$
_pF_q(\underline{a};\underline{b};x;\alpha)=\sum_\lambda \frac{(\underline{a})_\lambda}
{(\underline{b})_\lambda} \alpha^{|\lambda|} J_\lambda^*(x;\alpha).
$$
Since
$$
\omega_\alpha P_\lambda(x;\alpha) = Q_{\lambda'}(x;\alpha^{-1})
$$
and $J_\lambda^*(x;\alpha)=P_\lambda(x;\alpha)/h'_\lambda(\alpha)$, it follows that
\begin{align*}
\omega_\alpha J_\lambda^*(x;\alpha) & = Q_{\lambda'}(x;\alpha^{-1})/h'_\lambda(\alpha)
\\
&=\alpha^{-|\lambda|} Q_{\lambda'}(x;\alpha^{-1})/h_{\lambda'}(\alpha^{-1})
\\
&=\alpha^{-|\lambda|} J_{\lambda'}^*(x;\alpha^{-1}).
\end{align*} 
Next,
\begin{align*}
(a;\alpha)_\lambda &= \prod_{i=1}^n (a-(i-1)\alpha^{-1})_{\lambda_i}
\\
&=\prod_{(i,j)\in\lambda} (a-(i-1)\alpha^{-1} + j-1)
\\
&=(-\alpha^{-1})^{|\lambda|}\prod_{(i,j)\in\lambda'} (-a\alpha-(i-1)\alpha + j-1)
\\
&=(-\alpha^{-1})^{|\lambda|} (-a\alpha;\alpha^{-1})_{\lambda'}
\end{align*}

Hence
\begin{align*}
\omega_\alpha\; \mbox{} _pF_q(\underline{a};\underline{b};x;\alpha) &=
\sum_\lambda \frac{(-\alpha \underline{a};\alpha^{-1})_{\lambda'}}
                                  {(-\alpha \underline{b};\alpha^{-1})_{\lambda'}}
                                  (-\alpha^{-1})^{(p-q)|\lambda|} J_{\lambda'}^* (x;\alpha^{-1})
\\
&= \mbox{} _pF_q(-\alpha \underline{a};-\alpha \underline{b}; (-1)^{p-q} \alpha^{q-p+1} x; \alpha^{-1}).
\end{align*}
\end{proof}

\newpage

\section{Gauss \& Saalschutz again}
From \eqref{eq_6.21} we have
\begin{equation}
_2F_1(a,b;c;y;\alpha) = 
\frac{\Gamma_n(c;\alpha)}{c_n'(\alpha) \Gamma_n(a;\alpha) \Gamma_n(c-a;\alpha)}
\left(\int_0^1 \right)^n \mbox{} _1F_0(b;x,y;\alpha) |x|^{a-p} |1-x|^{c-a-p} 
|\Delta(x)|^{2k} \mathrm{d}x.
\label{eq_7.1} 
\end{equation}
If we interchange $x$ and $1-x$ in the integrand, and use \eqref{eq_6.29}, we shall obtain, just as in 
\S3,
\begin{equation} 
_2F_1(a,b;c;y;\alpha) = |1-y|^{-b} \mbox{}_2F_1(c-a,b;c;-y(1-y)^{-1};\alpha). 
\label{eq_7.2} 
\end{equation}

Reiterating then with $b,c-a$ in place of $a,b$ will give Euler's relation
\begin{equation} 
\label{eq_7.3}
_2F_1(a,b;c;y;\alpha)=|1-y|^{c-a-b}\,_2F_1(c-a,c-b;c;y;\alpha). 
\end{equation}
Setting $y=1$ in \eqref{eq_7.1} gives
\begin{flalign}
\label{eq_7.4}
\mbox{(Gauss)}\quad\quad _2F_1(a,b;c;1_n,\alpha) = \frac{\Gamma_n(c,\alpha)\Gamma_n(c-a-b,\alpha)}{\Gamma_n(c-a;\alpha)\Gamma_n(c-b;\alpha)}&&
\end{flalign}
and \eqref{eq_7.3} (which of course assumes \eqref{eq_6.25}) will then give, just as in \S4, 
\begin{flalign}
\label{eq_7.5}
\mbox{(Saalschutz)} \quad\quad \mbox{Let  } \underbar{$a$}= (a_1, a_2, a_3), \underbar{$b$}=(b_1, b_2)&&
\end{flalign}
where 
\begin{itemize}
\item[(i)] $\left(\sum_1^3 a_i\right)+p=\sum b_i$
\item[(ii)] one of the $a_i$ is a negative integer.
\end{itemize}
Then
$$_3F_2(\underbar{$a$}; \underbar{$b$};1_n;\alpha)=\prod_I \Gamma_n(b_1 -a_I; \alpha)^{(-1)|I|+1},$$
where the product is taken over the 8 subsets $I$ of $\{1,2,3\}$, and $a_I=\sum_{i\in I}a_i$.
\newpage

\section{Bessel functions \& Hankel transform}
We define (following Herz) the Bessel function of order $a$ by

\begin{eqnarray} 
\label{eq_8.1}
A_a(x)
&=&
A_a^{(n)}(x;\alpha)=\Gamma_n(a+p;\alpha)^{-1} \mbox{}_0F_1(a+p;-x;\alpha)
\\ 
&=&
\sum_\lambda \frac{(-\alpha)^{|\lambda|}J_\lambda ^*(x;\alpha)}{\Gamma_n (a+p;\lambda;\alpha)},
\nonumber
\end{eqnarray}
and the Bessel kernel by

\begin{eqnarray} 
\label{eq_8.2} 
A_a(x,y)
&=&
A_a^{(n)}(x;y;\alpha)
\\
&=&
\Gamma_n(a+p;\alpha)^{-1} \mbox{}_0F_1(a+p;-x,y;\alpha)  
\nonumber
\\
&=&
\sum_\lambda \frac{(-\alpha)^{|\lambda|}\Omega_\lambda(x;\alpha)J_\lambda^*(y;\alpha)}{\Gamma_n(a+p;\lambda;\alpha)}
\nonumber 
\end{eqnarray}

\noindent
\underline{Notation} : Let
$$
\mathrm{d}\mu(x)=c_n'(\alpha)^{-1}|\Delta(x)|^{2k}\mathrm{d}x$$
 $$e(x,y)=\mbox{}_0F_0(x,y;\alpha)$$
so that $e(1,y)=e^{\tr (y)}$ and

\begin{equation} 
\label{eq_8.3}
e(x,y)e(1,y)=e(x+1,y).
\end{equation}
Assume conjecture \eqref{conjC}.
\begin{equation}
\int e(-x,y)|x|^a \Omega_{\lambda}(x)\mathrm{d}\mu(x)=\Gamma_n (a+p;\lambda)|y|^{-a-p} \Omega_{\lambda}(y^{-1})
\label{conjC}
\tag{C}
\end{equation}
(here and in future we shall write $\int$ for
$(\int^{\infty}_0)^n$). 

Then we have
\begin{equation} 
\label{eq_8.4}
\int e(-x,y)A_a(x,z)|x|^a\mathrm{d}\mu(x)=|y|^{-a-p}e(-y^{-1},z)
\end{equation}
from \eqref{conjC} and the definition \eqref{eq_8.2}, by integrating term by term.

Also from \eqref{eq_6.21} we have
\begin{equation} 
\label{eq_8.5}
\frac{1}{\Gamma_n(b;\alpha)} \left(\int^1_0\right)^n A_a(x,y;\alpha)|x|^a|1-x|^{b-p}\mathrm{d}\mu(x)=A_{a+b}(y)
\end{equation}
-- a sort of addition theorem.

Likewise, for the function $_1F_1$ we have
\begin{equation} 
\label{eq_8.6}
\int e^{-\tr x}\mbox{} _1F_1(a;b;x,y;\alpha) |x|^{b-p} |1-x|^{a-b-p}\mathrm{d}\mu(x)
=B_n(b,a-b;\alpha)|1-y|^{-a},
\end{equation} 
where $B_n$ is the beta-function, and

\begin{equation}
_1F_1(a;b;y;\alpha)=
B_n(a,b-a;\alpha)^{-1} 
\left(\int^1_0\right)^ne(x,y)|x|^{a-p}|1-x|^{b-a-p}\mathrm{d}\mu(x).
\label{eq_8.7}
\end{equation}
Replacing $x$ by $1-x$ and using \eqref{eq_8.3}, we obtain Kummer's relation

\begin{equation} 
\label{eq_8.8}
_1F_1(a;b;y;\alpha)=e^{\tr y}\mbox{}_1F_1(b-a;b;-y;\alpha)
\end{equation}
Again following Herz, we define the \underline{Laguerre function}

\begin{equation} 
\label{eq_8.9}
L^{(a)}_b(x;\alpha)=\frac{\Gamma_n(a+b+p;\alpha)}{\Gamma_n(a+p;\alpha)}\mbox{}_1F_1(-b;a+p;x;\alpha)
\end{equation}
or equivalently, by \eqref{eq_8.8}

\begin{eqnarray} 
\label{eq_8.10}
L^{(a)}_b(x;\alpha)
&= &
\frac{\Gamma_n(a+b+p;\alpha)}{\Gamma_n(a+p;\alpha)}e^{\tr (x)}  \mbox{}_1F_1(a+b+p;a+p;-x;\alpha)
\\
&= &
e^{\tr (x)}\sum_\lambda \frac{\Gamma_n(a+b+p;\lambda)}{\Gamma_n(a+p;\lambda)}(-\alpha)^{|\lambda|}J^*_\lambda(x;\alpha),
\nonumber
\end{eqnarray}
\eqref{eq_6.20}, \eqref{eq_6.21} applied to the right hand side of \eqref{eq_8.10} give respectively

\begin{equation} 
\label{eq_8.11}
L^{(a)}_b(y;\alpha)=e^{\tr y}\int e^{-\tr x} A_a(x,y)|x|^{a+b}\mathrm{d}\mu(x),
\end{equation}
a formula that we shall generalize later, and

\begin{equation} 
\label{eq_8.12}
L^{(a)}_b(y;\alpha)=\Gamma_n(-b,\alpha)^{-1} e^{\tr y}\left(\int^1_0\right)^n e(-x,y)|x|^{a+b}|1-x|^{-b-p}\mathrm{d}\mu(x).
\end{equation}

\subsection*{Hankel transform}
The Hankel transform of index $a$ is defined by

\begin{equation} 
\label{eq_8.13}
(H_af)(y)=\int A_a(x,y)|x|^a f(x)\mathrm{d}\mu(x).
\end{equation}
If we define 
$$e_y(x)=e(-x,y)$$
then \eqref{eq_8.4} says that

\begin{equation} 
\label{eq_8.14}
H_a(e_y)=|y|^{-a-p}e_{y^{-1}}.
\end{equation}
From this it follows that $H^2_a=1$ on the linear space $U$ spanned by the functions $e_y, y>0$. The space $U$ is a subspace of
$$L_a^2=\left\{f : \int|f(x)|^2|x|^a\mathrm{d}\mu(x)<\infty\right\}$$
which carries a scalar product
$$\langle f,g \rangle_a=\int f(x)g(x)|x|^a\mathrm{d}\mu(x)$$
(if we are dealing with real-valued functions). Presumably $e_y\in L^2_a$ for $y>0$, but this remains to be proved, like the conjecture (C) for arbitrary $\alpha$. 

Let us compute the scalar product 
\begin{eqnarray*}
\langle e_y,e_z\rangle _a
&=&
\int e(-x,y)e(-x,z)|x|^a\mathrm{d}\mu(x)
\\
&=&
\int e(-x,y)\left(\sum_{\lambda}(-\alpha)^{|\lambda|}\Omega_{\lambda}(x)J^*_{\lambda}(z)\right)|x|^a \mathrm{d}\mu(x).
\end{eqnarray*}
On integrating term by term and using (C) we obtain
$$|y|^{-a-p}\Gamma_n(a+p)\sum_{\lambda}(-\alpha)^{|\lambda|}(a+p)_{\lambda}\Omega_{\lambda}(y^{-1})J^*_{\lambda}(z)$$
so that

\begin{eqnarray}
\label{eq_8.15}
\langle e_y,e_z\rangle _a
&=&
|y|^{-(a+p)}\Gamma_n (a+p)\cdot \mbox{} _1F_0(a+p;-y^{-1},z)
\\
&=&
|z|^{-(a+p)}\Gamma_n (a+p)\cdot \mbox{} _1F_0(a+p;-y,z^{-1})
\nonumber
\end{eqnarray}
on interchange of $y$ and $z$. From this and \eqref{eq_8.14} we have

\begin{equation} 
\label{eq_8.16}
\langle H_ae_y, H_ae_z \rangle _a=\langle e_y, e_z\rangle _a
\end{equation}
so that $H_a$ is an isometry on $U$. Moreover, if we assume (why not?) that the Laplace transform is injective, then $\langle f,e_y\rangle _a=0$ for all $y>0$ implies that $f=0$ in $L^2_a$, and therefore $U$ is dense in the Hilbert space $L^2_a$. So finally $H_a$ extends to a self-adjoint involution on $L^2_a$.

For $f\in L^2_a$, let $\phi_f$ denote the Laplace transform of $|x|^a f(x)$: that is to say,
$$ \phi_f(y)=\int e(-x,y)|x|^af(x) \mathrm{d}\mu (x)=\langle e_y,f\rangle _a $$

\begin{flalign} 
\mbox{Suppose that $\phi_g(y)$ is an absolutely convergent integral for all $y>0$.}&& 
\label{eq_8.17}
\end{flalign}
\begin{enumerate}
\item[(i)] If $g=H_af$ then
\begin{equation}
 \phi_g(y)=|y|^{-a-p}\phi_f(y^{-1}).
\tag{8.17.1}
\label{eq_8.17.1}
\end{equation}
\item[(ii)] Conversely, if \eqref{eq_8.17.1} holds, then $g\in L^2_a$ and $g=H_af$.
\end{enumerate}

\begin{proof}
Let $g_1=H_af$, then
\begin{eqnarray*}
\phi_{g_1}(y)
&=&
\langle e_y,g_1\rangle _a=\langle H_ae_y,f\rangle _a \mbox{ \quad by \eqref{eq_8.16}}
\\
&=&
|y|^{-a-p}\langle e_{y^{-1}},f\rangle _a \mbox{ \quad by \eqref{eq_8.14}}
\\
&=&
|y|^{-a-p}\phi_f(y^{-1}).
\end{eqnarray*}
This proves (i). In case (ii) we have $\phi_g=\phi_{g_1}$, whence $\langle e_y,g\rangle _a=\langle e_y,g_1\rangle _a$ for all $y>0$; hence $g=g_1$ on $L_a^2$.
\end{proof}

\subsection*{Laguerre polynomials}
 As an example of \eqref{eq_8.17}, let us take
\begin{eqnarray*}
f(x)
&=&
e^{-\tr (x)}J^*_{\lambda}(x;\alpha) \\
g(y)
&=&
(H_af)(y)=e^{-\tr (y)}L_{\lambda}^{(a)}(y) \mbox{, say.}
\end{eqnarray*}
We shall see presently that the $L_{\lambda}^{(a)}$ are symmetric polynomials, called Laguerre polynomials (they depend on $\alpha$ as well as $a$).

We have
\begin{eqnarray*}
\phi_f(y)
&=&
\int e(-x,y)e^{-\tr x} |x|^a J^*_{\lambda} (x;\alpha)\mathrm{d}\mu(x)\\
&=&
\int e(-x, 1+y) |x|^a J^*_{\lambda}(x;\alpha)\mathrm{d}\mu(x)\\
&=&
\Gamma_n(a+p;\lambda)|1+y|^{-a-p}J^*_{\lambda}((1+y)^{-1};\alpha)
\end{eqnarray*}
if \eqref{conjC} is true. By \eqref{eq_8.17} this is equal to
\begin{eqnarray*}
|y|^{-a-p}\phi_g(y^{-1})&=&|y|^{-a-p} \int e(-x,y^{-1})e^{-\tr x}L^{(a)}_{\lambda}(x)|x|^a \mathrm{d}\mu(x)
\\
& =&|y|^{-a-p} \int e(-x, 1+y^{-1})|x|^a L^{(a)}_{\lambda}(x)\mathrm{d}\mu(x)
\end{eqnarray*}
so that we have
$$
\int e(-x,1+y^{-1})
|x|^a L^{(a)}_{\lambda} (x)\mathrm{d}\mu(x)
= \Gamma_n(a+p;\lambda)\left|\frac{y}{1+y}\right|^{a+p}J^*_{\lambda}((1+y)^{-1};\alpha)
$$
Put $z=(1+y^{-1})^{-1}=y/(1+y)$, so that $(1+y)^{-1}=1-z$ ; then we have

\begin{equation}
\int e(-x,z^{-1})|x|^{a}L^{(a)}_{\lambda}(x)\mathrm{d}\mu(x)=\Gamma_n(a+p;\lambda)|z|^{a+p}J^*_{\lambda}(1-z), 
\label{eq_8.18}
\end{equation}
i.e., the Laplace transform of $|x|^{a}L^{(a)}_{\lambda}(x)$ is
$$\Gamma_n(a+p;\lambda)|z|^{-a-p}J^*_{\lambda}(1-z^{-1}).$$

But we have proved that
$$\Omega_{\lambda}(1-z)=\sum_{\mu \subset \lambda} (-1)^{|\mu|} \varepsilon (J_{\lambda/\mu})\Omega_{\mu}(z)$$
so that we have

\begin{equation} 
J_{\lambda}^*(1-z)=\sum_{\mu \subset \lambda} (-1)^{|\mu|} a^{(n)}_{\lambda/\mu} J^*_{\mu}(z)
\label{eq_8.19}
\end{equation}
with
\begin{eqnarray}
\label{eq_8.20}
a^{(n)}_{\lambda/\mu}
&=& \varepsilon(J_{\lambda/\mu})J_{\lambda}^*(1_n)/J^*_{\mu}(1_n) \\
&=& \varepsilon(P_{\lambda/\mu})Q_{\lambda}(1_n)/Q_{\mu}(1_n). 
\nonumber
\end{eqnarray}
It follows that
\begin{eqnarray*}
\Gamma_n(a+p;\lambda)^{-1} \int e(-x,z^{-1})|x|^a L^{(a)}_{\lambda}(x)\mathrm{d}x &=& |z|^{a+p} \sum_{\mu \subset \lambda} (-1)^{|\mu|}a_{\lambda/\mu}J^*_{\mu}(z) \\
&=& \sum_{\mu \subset \lambda} \frac{(-1)^{|\mu|}}{\Gamma_n(a+p;\mu)} a_{\lambda/\mu} \int e(-x,z^{-1})|x|^a J^*_{\mu}(x)\mathrm{d}\mu(x)
\end{eqnarray*}
and therefore, on the assumption that the Laplace transform is injective, we should have

\begin{equation}
L^{(a)}_{\lambda}(x;\alpha)=\sum_{\mu \subset \lambda} (-1)^{|\mu|}a^{(n)}_{\lambda/\mu}(\alpha)\frac{(a+p;\alpha)_\lambda}{(a+p;\alpha)_\mu}J^*_{\mu}(x;\alpha)
\label{eq_8.21}
\end{equation}
a symmetric (but not homogeneous) polynomial in $x_1,\ldots,x_n$. 

\noindent \underline{Remark}. From the definition
$$
L^{(a)}_{\lambda}(y)=e^{\tr (y)} \int e^{-\tr (x)} A_a(x,y)|x|^aJ_{\lambda}^*(x;\alpha)\mathrm{d}\mu(x)
$$
and \eqref{eq_8.11} it follows that (up to a scalar factor) $L_b^{(a)}$ is the same as $L^{(a)}_{(b,\ldots,b)}$.

Next, we have
\begin{equation}
e(-x^{-1},y)=|x|^{a+p}e^{-\tr y} \sum_{\lambda} \alpha^{|\lambda|}\Omega_{\lambda}(1-x)L^{(a)}_{\lambda}(y)
\label{eq_8.22}
\end{equation}

\begin{proof} 
The right hand side of \eqref{eq_8.22} is the Hankel transform (with respect to $y$) of the series
\begin{eqnarray*}
&&|x|^{a+p}e^{-\tr z} \sum_{\lambda} \alpha^{|\lambda|}\Omega_{\lambda}(1-x)J^*_{\lambda}(z) \\
&=&|x|^{a+p}e^{-\tr z} e(1-x,z)=|x|^{a+p}e(-x,z) \\
&=&|x|^{a+p}e_x(z).
\end{eqnarray*}
The result now follows from \eqref{eq_8.14}.
\end{proof}

\begin{flalign}
\label{eq_8.23}
\mbox{The functions}&& 
\end{flalign}
$$f_{\lambda}(x)=e^{-\tr x} L_{\lambda}^{(a)}(2x)$$
are eigenfunctions of the Hankel transform $H_a$ (the eigenvalue being $(-1)^{|\lambda|}$).

\begin{proof} We have
\begin{eqnarray*}
\phi_{f_{\lambda}}(y) &=& \int e(-x,y)e^{-\tr (x)}|x|^a L^{(a)}_{\lambda}(2x)\mathrm{d}\mu(x) \\
&=& \int e(-x,1+y)|x|^aL^{(a)}_{\lambda}(2x)\mathrm{d}\mu(x) \\
&=& 2^{-n(a+p)} \int e\left(-x, \frac{1}{2}(1+y)\right)|x|^aL^{(a)}_{\lambda}(x)\mathrm{d}\mu(x)
\end{eqnarray*}
(since $|x|^{-p}\mathrm{d}\mu(x)$ is translation-invariant). \\
Put $\frac{1}{2}(1+y)=z^{-1}$, so that $1-z=1-\frac{2}{1+y}=\frac{y-1}{y+1}$; then from \eqref{eq_8.18} we have
$$
\phi_{f_{\lambda}} (y)=\Gamma_n(a+p;\lambda) |1+y|^{-a-p}J^*_{\lambda} \left(\frac{y-1}{y+1}\right)$$
from which it follows that
$$\phi_{f_{\lambda}}(y^{-1})=|y|^{a+p}(-1)^{|\lambda|}\phi_{f_{\lambda}}(y)$$
and hence, by \eqref{eq_8.17}, that $H_af_{\lambda}=(-1)^{|\lambda|}f_{\lambda}$.
\end{proof}

\underline{Question}: Is it the case that $\langle f_{\lambda},f_{\mu}\rangle _a=0$ if $\lambda \neq \mu$, i.e., is
$$
\int e^{-\tr x} L^{(a)}_{\lambda}(x)L^{(a)}_{\mu}(x)|x|^a \mathrm{d}\mu(x)=0
$$
if $\lambda \neq \mu$?

The answer is yes. From \eqref{eq_8.22} we have, on multiplying both sides by $|y|^aL^{(a)}_{\mu}(y)$ and integrating term by term

\begin{align*}
\int e(-x^{-1},y)L^{(a)}_{\mu} (y)|y|^a \mathrm{d}\mu (y) & =  |x|^{a+p} \sum_{\lambda} \alpha^{|\lambda|}\Omega_{\lambda}(1-x)\int e^{-\tr y} L^{(a)}_{\lambda}(y)L^{(a)}_{\mu}(y)|y|^a \mathrm{d}\mu(y) \\
& =  2^{n(a+p)} |x|^{a+p} \sum_{\lambda} \alpha^{|\lambda|} \Omega_{\lambda}(1-x)\langle f_{\lambda}, f_{\mu}\rangle _a
\tag{1}
\end{align*}

$$\left( \mbox{Since}\mbox{ } \mbox{  }\mbox{  }\langle f_{\lambda}, f_{\mu}\rangle _a=\int e^{-2\tr x} L_{\lambda}^{(a)}(2x)L^{(a)}_{\mu}(2x)|x|^a \mathrm{d}\mu(x) = 2^{-n(a+p)} \int e^{-\tr x} L^{(a)}_{\lambda}(x) L^{(a)}_{\mu}(x)|x|^a \mathrm{d}\mu(x) \right).$$

But on the other hand, by \eqref{eq_8.18}, we have
\begin{equation}
\int e(-x^{-1}, y)  L^{(a)}_{\mu}(y)|y|^a \mathrm{d}\mu (y)=\Gamma_n (a+p;\mu)|x|^{a+p}J^*_{\lambda}(1-x).
\tag{2}
\end{equation}

From (1) and (2) it follows that

\begin{equation}
\langle f_{\lambda}, f_{\mu}\rangle _a=0 \mbox{ }  \mbox{if}\mbox{ } \lambda  \neq  \mu
\label{eq_8.24}
\end{equation}
and
$$
2^{n(a+p)}\alpha^{|\mu|}\Omega_{\mu}|f_{\mu}|^2_a=\Gamma_n(a+p;\mu)J^*_{\mu}
$$
so that

\begin{eqnarray}
\label{eq_8.25}
|f_{\mu}|^2_a &=& 2^{-n(a+p)}\alpha^{-|\mu|}\Gamma_n (a+p;\mu )J^*_{\mu}(1_n)\\ 
&=& 2^{-n(a+p)} \frac{\Gamma_n(a+p;\mu)(nk)_{\mu}}{h_{\mu}(\alpha)h'_{\mu}(\alpha)}
\nonumber
\end{eqnarray}

Next we have
\begin{equation}
e^{\tr (y)}A_a(x,y)=\sum_{\lambda} \frac{\alpha^{|\lambda|}L^{(a)}_{\lambda}(x)\Omega_{\lambda}(y)}{\Gamma_n(a+p;\lambda)}
\label{eq_8.26}
\end{equation}

\begin{proof} We can expand the left hand side of \eqref{eq_8.26} in a series of Laguerre polynomials, say
\begin{equation}
e^{\tr (y)} A_a(x,y)=\sum_{\lambda}L^{(a)}_{\lambda}(x)u_{\lambda}(y)
\tag{$\ast$}
\end{equation}
To determine the coeficients, we use orthogonality \eqref{eq_8.24}: multiply both sides by $e^{-\tr (x)}L^{(a)}_{\mu}(x)|x|^a$ and integrate with respect to $x$.\\
The right hand side of ($\ast$) then becomes
$$
\alpha^{-|\mu|}\Gamma_n(a+p;\mu)J^*_{\mu}(1_n)u_{\mu}(y)
$$
and the left side becomes
$$
e^{\tr y} \int e^{-\tr x}A_a(x,y)L^{(a)}_{\mu}(x)|x|^a\mathrm{d}\mu(x)=J^*_{\mu}(y)
$$
by definition of the Hankel transform. So we have
$$
u_{\mu}(y)=\alpha^{|\mu|}\Omega_{\mu}(y)/\Gamma_n(a+p;\mu)
$$
which proves \eqref{eq_8.26}.
\end{proof}

\begin{flalign}
\sum_{\lambda}\alpha^{|\lambda|}\Omega_{\lambda}(x)L^{(a)}_{\lambda}(y)=|1-x|^{-a-p}e\left(-\frac{x}{1-x},y\right).&&
\label{eq_8.27}
\end{flalign}

\begin{proof} This follows from \eqref{eq_8.22} on replacing $x$ by $1-x$, since
$$
e^{\tr y}\, e\left(-\frac{1}{1-x}, y\right)=e\left(1-\frac{1}{1-x},y\right)=e\left(\frac{-x}{1-x},y\right).
$$
\end{proof}

From \eqref{eq_8.25} we have
$$
\Gamma_n(a+p;\lambda)=\frac{2^{n(a+p)}\alpha^{|\lambda|}|f_{\lambda}|^2_a}{J^*_{\lambda}(1_n)}
$$
and therefore

\begin{equation}
A_a(-x,y)=\sum_{\lambda}\frac{\alpha^{|\lambda|}\Omega_{\lambda}(x)J_{\lambda}^{*}(y)}{\Gamma_n(a+p;\lambda)}=2^{-n(a+p)}\sum_{\lambda}\frac{J_{\lambda}^*(x)J_{\lambda}^{*}(y)}{|f_{\lambda}|^2_a}.
\label{eq_8.28}
\end{equation}
\bigskip

Another approach to the Laguerre polynomials is as eigenfunctions of a differential operator (i.e., as limits of Jacobi polynomials: $$ L^{(a)}_{\lambda}(x;\alpha)=\lim_{b \rightarrow \infty} G^{(a,b)}_{\lambda}\left(\frac{x}{b};\alpha\right).\;)$$
If we replace $x$ by $b^{-1}x$ in the operator $E_{a,b}$, \& let $b\rightarrow \infty$, we get
$$E_af=-|x|^{-a}e(x)\Delta(x)^{-2k}\sum_{i=1}^n D_i(x_i|x|^ae(-x)\Delta(x)^{2k}D_if)$$
which when computed explicitly gives
\begin{eqnarray*}
-E_a
&=& 2U_{12}+2kV_1+(a+1)U_{01}-U_{11} \\
&=& \Box _1+(a+p)U_{01}-U_{11}
\end{eqnarray*}
and hence
$$E_a\Omega_{\lambda}=|\lambda|\Omega_{\lambda}-\sum_{\substack{\mu\subset\lambda\\|\lambda-\mu|=1}} \binom{\lambda}{\mu}(a+p+\rho(\lambda/\mu))\Omega_{\mu}$$
in which $a+p+\rho(\lambda/\mu)=(a+p)_{\lambda}/(a+p)_{\mu}$. \\
So $L^{(a)}_{\lambda}(x;\alpha)$ satisfies
$$E_aL^{(a)}_{\lambda}=|\lambda|L^{(a)}_{\lambda}$$
and if 
$$L_{\lambda}^{(a)}(x)=\sum_{\mu\subset\lambda}(-1)^{|\mu|}u_{\lambda\mu}\Omega_{\mu}$$
we obtain the recurrence relation
$$ |\lambda-\nu|u_{\lambda \nu}=\sum_{\mu}\binom{\mu}{\nu}\frac{(a+p)_{\mu}}{(a+p)_\nu}u_{\lambda\mu}$$
summed over $\mu$ such that $\lambda\supset\mu\supset \nu$ and $|\mu-\nu|=1$.\\
If $|\lambda-\nu|=r$ this gives explicitly (if $u_{\lambda\lambda}=1$)

$$u_{\lambda \nu}=\frac{1}{r!}\sum_{T}\prod^r_{i=1}\binom{\lambda^{(i-1)}}{\lambda^{(i)}}
\cdot
\frac{(a+p)_{\lambda}}{(a+p)_\nu}$$
summed over all standard tableaux
$$T: \quad \lambda=\lambda^{(0)}\supset\lambda^{(1)}\supset\cdots\supset\lambda^{(r)}=\nu.$$
In other words,
$$ u_{\lambda \nu} = \binom{\lambda}{\nu}\frac{(a+p)_{\lambda}}{(a+p)_\nu}$$
and therefore
$$ L^{(a)}_{\lambda}(x;\alpha)=\sum_{\mu\subset\lambda}(-1)^{|\mu|} \binom{\lambda}{\mu} \frac{(a+p)_{\lambda}}{(a+p)_\mu}\Omega_{\mu}(x;\alpha)$$
agreeing with \eqref{eq_8.21}, apart from the normalizing factor $J^{*}_{\lambda}(1_n)$. \\
So we have another proof of orthogonality.\\
\indent
In the 1-variable case we can recover the Jacobi polynomials from the Laguerre polynomials by Laplace transform (i.e., passing from $_1F_1$ to $_2F_1$)
$$G^{(a,b)}_n(y)=\int^{\infty}_0e^{-x}L_n^{(a)}(xy) \cdot \frac{x^{n+a+b}}{(n+a+b)!}\mathrm{d}x.$$
Is there an analogue of this in $n$ variables?
\newpage

\section*{Appendix: proof of \eqref{eq_6.15}}
let $\varepsilon:\Lambda\rightarrow {\mathbb Q}$ be the specialization defined by
$$\varepsilon(p_1)=1, \varepsilon(p_r)=0 \;\mbox{ if }\; r>1.$$
(Thus $\varepsilon(H(t))=\varepsilon(E(t))=e^t$.) 

\begin{unprop} Let $\lambda$ be a partition of length $\leq n$. Then

\begin{equation}
\Omega_{\lambda}(1_n+x)=\sum_{\mu\subset\lambda}\varepsilon(J_{\lambda/\mu})\Omega_{\mu}(x).
\tag{1}
\end{equation}
\end{unprop}

\begin{proof} Let $D_n=\sum^{n}_{i=1}\partial/\partial x_i ,$ then
$$\Omega_{\lambda}(1_n+x)=\sum_{r\geq0}\frac{1}{r!}D^r_n\Omega_{\lambda}(x)$$
so that (1) $\Leftrightarrow$
\begin{equation}
\frac{1}{r!}D_n^r\Omega_{\lambda}=\sum_{\substack{\mu\subset\lambda\\|\lambda-\mu|=r}}\varepsilon(J_{\lambda/\mu})\Omega_{\mu}\quad\quad\quad(r\geq 0).
\tag{1$'$}
\end{equation} 
We shall prove (1$'$) by induction on $r$. For the moment assume that it is true for $r=1$, i.e., that
\begin{equation}
D_n\Omega_{\lambda}=\sum_{\substack{\mu\subset\lambda\\|\lambda-\mu|=1}}\varepsilon(J_{\lambda/\mu})\Omega_{\mu}.
\tag{2}
\end{equation}
From (2) it follows by induction on $r$ that
\begin{equation}
D^r_n\Omega_{\lambda}=\sum\prod^r_{i=1}\varepsilon(J_{\mu^{(i-1)}/\mu^{(i)}})\Omega_{\mu}
\tag{3}
\end{equation}
summed over all chains of partitions
$$\lambda=\mu^{(0)}\supset\mu^{(1)}\supset\cdots\supset\mu^{(r)}=\mu$$
such that $|\mu^{(i-1)}-\mu^{(i)}|=1$ for $1\leq i \leq r$.
\end{proof}

\begin{lemmaA} Let $\mu \subset\lambda, |\lambda - \mu|=r$. Then
\begin{equation}
\varepsilon(J_{\lambda/\mu})=\frac{1}{r!}\sum \prod^r_{i=1}\varepsilon(J_{\mu^{(i-1)}/\mu^{(i)}})
\tag{4}
\end{equation}
summed over chains of partitions as above.
\end{lemmaA}

{\begin{proof} We have
$$ J_{\lambda/\mu}(x_1,\ldots,x_r)=\sum\prod^r_{i=1}J_{\mu^{(i-1)}/\mu^{(i)}}(x_i)$$
in which the coefficient of $x_1,\ldots, x_r$ is
$$a=\sum\prod^r_{i=1}\varepsilon(J_{\mu^{(i-1)}/\mu^{(i)}})$$
On the other hand, $\varepsilon(J_{\lambda/\mu})$ is the coefficient of $p_1^r$ in $J_{\lambda/\mu}$, and hence $a=r!\varepsilon(J_{\lambda/\mu})$ (since $x_1\cdots x_r$ occurs in $P_\lambda$ only for $\lambda=(1^r)$).
\end{proof}}

From (3) and (4) we have (1$'$). Hence it remains to prove (2). We rewrite (2) in terms of the $P_\lambda$:
\begin{eqnarray*}
D_n(P_\lambda)&=&\sum_{\substack{\mu\subset\lambda\\|\lambda-\mu|=1}}\varepsilon(J_{\lambda/\mu})\frac{P_\lambda(1_n)}{P_\mu(1_n)}P_k \\
&=& \sum_{\mu\subset\lambda}\varepsilon(P_{\lambda/\mu})\frac{J_\lambda(1_n)}{J_\mu(1_n)}P_\mu 
\end{eqnarray*}
Now
$$J_\lambda(1_n)/J_\mu(1_n)=n+a'(s)\alpha-l'(s)$$
where $\{s\}=\lambda-\mu$, so that

\begin{equation}
D_nP_\lambda=\sum_s\varepsilon(P_{\lambda/\mu})(n+a'(s)\alpha-l'(s))P_\mu
\tag{2$'$}
\end{equation}
summed over the corner squares $s\in\lambda$, with $\lambda-\{s\}=\mu$. 

Now $D_n$ is the derivation of  $\Lambda$ defined by
$$D_nP_r=rP_{r-1}(r\geq2), \quad D_nP_1=n.$$
Both sides of (2$'$) are linear in $n$, \& we may replace $n$ by an indeterminate $X$: define a derivation $D_X$ by
$$D_XP_r=rP_{r-1} \mbox{  } (r\geq2), D_XP_1=X.$$
Then we have to prove that

\begin{equation}
D_X P_\lambda=\sum_s\varepsilon(P_{\lambda/\mu})(X+a'(s)\alpha-l'(s))P_\mu.
\tag{2$''$}
\end{equation}
Both sides are linear in $X$, so it is enough (for fixed $\lambda,\alpha$) to prove that they agree for two values of $X$. We shall do this (a) for $X=\infty$ (b) for $X=l(\lambda)=n$ say. 

From the definition of $D_X$ it is clear that

\begin{equation}
\lim_{X\rightarrow\infty}\frac{D_X}{X}=\frac{\partial}{\partial p_1}
\tag{5}
\end{equation}

\begin{lemmaA} We have
\begin{equation}
\frac{\partial}{\partial p_1}P_\lambda=P_{\lambda/(1)}
\tag{6}
\end{equation}
\end{lemmaA}

\begin{proof} We have $(f, g \in \Lambda)$
$$\left\langle \frac{\partial f}{\partial p_1}, g\right\rangle_\alpha = \alpha^{-1}\langle f,p_1 g \rangle_\alpha$$
(by linearity it is enough to check this when $f,g$ are monomials in the $p$'s, \& then it is obvious). Since $\alpha^{-1}p_1=Q_{(1)}$ we have
$$ \left\langle \frac{\partial}{\partial p_1}P_{\lambda}, g \right\rangle_\alpha = \langle P_{\lambda}, Q_{(1)}g\rangle_\alpha = \langle P_{\lambda/(1)},g\rangle$$ 
which proves (6).
\end{proof}

Now let $y$ be a single variable; then
$$P_{\lambda}(x;y)=\sum_{\mu}P_{\mu}(x)P_{\lambda/\mu}(y)=\sum_v P_v(y) P_{\lambda/v}(x).$$
By considering the coefficient of $y$ in these two expressions we have
\begin{equation}
P_{\lambda/(1)}(x)=\sum_{\mu} P_{\mu}(x)P_{\lambda/\mu}(1).
\tag{7}
\end{equation}

From (5), (6) and (7) it follows that the coefficient of $X$ in $D_X P_\lambda$ is equal to
$$\sum_{\substack{\mu\subset\lambda\\|\lambda-\mu|=1}} \quad \varepsilon(P_{\lambda/\mu})P_\mu$$
so that (2$''$) is true for $X=\infty$. 

 It remains to prove (2$''$) for $X=n=l(\lambda)$. Let $\lambda_{*} =(\lambda_{1} -1,\ldots , \lambda_{n} -1)$ so that $P_{\lambda}=e_{n} P_{\lambda_{*}}$ and therefore
\begin{equation}
D_nP_\lambda=D_n(e_n P_{\lambda_{*}}) = e_{n-1}P_{\lambda_{*}}+e_n D_n P_{\lambda_{*}}
\tag{8}
\end{equation}

Now
$$ e_{n-1}P_{\lambda_*}=\sum \psi'_{\mu/\lambda_*}P_\mu$$
summed over $\mu\supset\lambda_*$ such that $\mu-\lambda_*$ is a vertical strip of length $(n-1)$, i.e.,  $\mu\subset\lambda, |\lambda-\mu|=1$. Also
$$\psi'_{\mu/\lambda_*}=b_\mu (\bar R_{\mu/\lambda_*})/b_{\lambda_*}(\bar R_{\mu/\lambda_*}) = b_\mu(R)/b_{\lambda_*}(R),$$
where:  $\lambda-\mu$ is the single square $s=(i,\lambda_i)$, and $R=\{(i,j): 1\leq j \leq \lambda_i -1\}$ is the $i$th row of $\mu$ (or $\lambda_*)$. Let $t=(1,\lambda_i)$ be the leftmost square in the row $R$, then $b_{\lambda_*}(R)=b_\lambda (R)/b_{\lambda}(t)$. 

Hence $$ \psi'_{\mu/\lambda_*}=\frac{b_\mu(R)}{b_\lambda(R)}b_\lambda(t).$$
Moreover (2$''$) now takes the form
\begin{equation}
D_X P_\lambda = \sum_\mu \varepsilon(P_{\lambda/\mu})h_\lambda(t)P_\mu
\tag{2$'''$}
\end{equation}
because with $s=(i,\lambda_i)$ we have
$$n+a'(s)\alpha-l'(s)=n+(\lambda_i -1)\alpha-i+1=a(t)\alpha+l(t)+1.$$

We proceed by induction on $\lambda$:  thus we have to show that
\begin{equation}
\varepsilon(P_{\lambda/\mu})h_\lambda(t)=\frac{b_\mu (R)}{b_\lambda (R)}b_\lambda (t) + \varepsilon(P_{\lambda_*/\mu_*})h_{\lambda_*}(t).
\tag{9}
\end{equation}

Now 
 $$ \varepsilon(P_{\lambda/\mu})=P_{\lambda/\mu}(1)=\psi_{\lambda/\mu}=\frac{b_\mu (R)}{b_\lambda (R)}$$
and
$$ \varepsilon(P_{\lambda_*/\mu_*})=\frac{b_\mu (R-\{t\})}{b_\lambda (R-\{t\})}=\frac{b_\mu (R)}{b_\lambda (R)} \cdot \frac{b_\lambda (t)}{b_\mu (t)};$$
thus (9) is equivalent to
$$h_\lambda (t)=b_\lambda(t)+\frac{b_\lambda(t)}{b_\mu(t)}h_\mu(t)$$
(since $h_{\lambda_*}(t)=h_\mu(t)$), that is to say to 
$$h'_\lambda (t)=1+h'_\mu (t)$$
But $h'_\lambda (t)=\alpha\lambda_i +n-i, \; h'_\mu (t)=\alpha(\lambda_i -1) + n-i$ \\
so that in fact
$$h'_\lambda (t) = \alpha + h'_\mu (t)$$
 -- we are still out by a factor $\alpha$.
\newpage

\section*{Laplace transform}
In the case $\alpha=2$ we defined the Laplace transform by
\begin{equation}
(Lf)(t)=\int_{\Sigma^+}e^{-\tr (st)}f(s)\mathrm{d}s \quad \quad \quad \quad (t\in\Sigma^+)
\nonumber
\tag{1}
\end{equation}
and showed that the Laplace transform of $|s|^{a-p}\Omega_\lambda(s)$ is $\Gamma_n(a;\lambda)|t|^{-a}\Omega_\lambda (t^{-1}).$

To write (1) as an integral over $X\,=\,\mathbb{R}^n_+$ we must introduce $$\int_K e^{-\tr (kxk'y)}\mathrm{d}k=e(-x,y)$$
where we have written $e$ in place of $_0F_0$. The Formula (1) then becomes
\begin{equation}
(Lf)(y)=\left(\int^\infty_0\right)^n e(-x,y)f(x)\mathrm{d}\mu(x)
\nonumber
\tag{2}
\end{equation}
where as usual 
$$\mathrm{d}\mu(x)=c'^{-1}_n |\Delta(x)|.$$
So for general $\alpha$ we take $(2)$ as our definition of the Laplace transform.

\begin{unconjecture}[C]
 The Laplace transform of $|x|^{a-p}\Omega_\lambda(x;\alpha) \mbox{ is } \Gamma_n(a;\lambda;\alpha)|y|^{-a}\Omega_\lambda(y^{-1};\alpha)$.
\end{unconjecture}

Assuming this we have
$$
\left(\int^\infty_0\right)^n e(-x,y^{-1})\mbox{}_pF_q(\underbar{$a$};\underbar{$b$};x)|x|^{a-p}\mathrm{d}\mu(x)=|y|^a\Gamma_n(a;\alpha)\mbox{ }_{p+1}F_q(\underbar{$a$}^+;b;y)$$
with the same notation as in \S 3.\footnote{See Additional observation1 at the end of this section.}

Conjecture (C) is true for $\alpha = 2,1,\frac{1}{2}$, surely; also for $\alpha=\infty$:--

\begin{unprop} We have

\begin{equation}
e(x,y;\infty)=\frac{1}{n!}\sum_{\omega\in S_n} e^{\langle x,\omega y\rangle}
\nonumber
\end{equation}
where $\langle x,y\rangle = \sum_{i=1}^n x_i y_i$        and $\omega y = (y_{\omega(1)},\ldots ,y_{\omega(n)}).$
\end{unprop}

\begin{proof} We have $P_\lambda (x;\infty)=m_\lambda (x),$ and
\begin{eqnarray*}
\alpha^{|\lambda|}J^*_\lambda (x;\alpha) &=& \frac{\alpha^{|\lambda|}}{h'_\lambda (\alpha)}P_\lambda (x) \\
&=& P_\lambda (x) \prod_{s\in \lambda} (a(s)+1+\alpha^{-1} l(s))^{-1}
\end{eqnarray*}
which is equal to $m_\lambda (x)/|\lambda|!$ when $\alpha = \infty$.

Hence
\begin{eqnarray*}
e(x,y;\infty) &=& \sum_{l(\lambda)\leq n}\quad \frac{m_\lambda (x)}{m_\lambda (1_n)}\cdot \frac{m_\lambda (y)}{\lambda!} \\
&=& \frac{1}{n!} \sum_{\alpha \in \mathbb{N}^n}\frac{y^d}{\alpha!}\sum_{\omega \in S_n} x^{\omega\alpha} \\
&=& \frac{1}{n!} \sum_{\omega \in S_n} \sum_{\alpha \in \mathbb{N}^n} \frac{x^{\omega\alpha}y^\alpha}{\alpha!} \\
&=& \frac{1}{n!} \sum_{\omega \in S_n} e^{\langle \omega x,y\rangle}.
\end{eqnarray*}
\end{proof}

Since $k=0$ we have $\mathrm{d}\mu(x) = \mathrm{d}x$ and therefore (as $p=1$)
\begin{eqnarray*}
\left(\int^\infty_0 \right)^n e(-x,y)|x|^{a-1}x^\beta \mathrm{d}\mu(x) &=& \frac{1}{n!}\sum_{\omega\in S_n}\prod^n_{i=1}\int^\infty_0 e^{-x_i y_{\omega(i)}} x_i^{a+\beta_i}\frac{\mathrm{d}x_i}{x_i} \\
&=& \frac{1}{n!}\sum_{\omega\in S_n}\prod^n_{i=1} y^{-(a+\beta_i)}_{\omega(i)}\Gamma(a+\beta_i) \\
&=& \Gamma_n(a;\beta)|y|^{-a}  \frac{1}{n!}\sum_{\omega\in S_n} y^{-\omega\alpha}.
\end{eqnarray*}

Hence the Laplace transform of $|x|^a m_\lambda (x)$ is
$$ \Gamma_n (a;\lambda) |y|^{-a}m_\lambda (y^{-1})$$
as required for Conjecture (C).

\underline{Remark} The series $e(x,y)$ converges absolutely for all $x,y$. For if $|x_i|\leq X$ for $1\leq i\leq n$, then we have $$|P_\lambda(x;\alpha)|\leq X^{|\lambda|}P_\lambda(1_n;\alpha)$$
because the coefficients of the monomials in $P_\lambda(x;\alpha)$ are all positive (we are assuming $\alpha \geq 0$). Hence $e(x,y;\alpha)$ is dominated (for $x \in [-X, X]^n$) by
$$\sum_\lambda X^{|\lambda|}\alpha^{|\lambda|} J_\lambda^* (y;\alpha) = e^{X\, {\rm trace}\,y}.$$
We should go on and show that (for $x \geq 0$) $e(-x,y)$ decreases exponentially for fixed $y$, so that we can confidently integrate against it.

From (C) it would follow, in particular, that
$$\left(\int^\infty_0\right)^n e(-x,y) e(x,z) |x|^{a-p}\mathrm{d}\mu(x) = |y|^{-a}\Gamma_n (a;\alpha)\mbox{ }_1 F_0 (a;y^{-1}, z)$$
\bigskip

As to Conjecture (C), assume the weaker 

\begin{unconjecture}[C$'$]  $\left( \int^\infty_0 \right)^n e(-x,y)|x|^{a-p}x^\lambda \mathrm{d}\mu(x)$ is of the form $|y|^{-a} \times$ polynomial in $y^{-1}$ with leading term $y^\lambda$.
\end{unconjecture}

Let $E$ be the operator
$$ Ef = \Delta(x)^{-2k}\sum^n_{i=1}x_i D_i (\Delta(x)^{2k}x_i D_i f).$$
Then (cf.\ ch.\ VI, \S10)
$$E\Omega_\lambda (x;\alpha) = \langle\lambda, \lambda +2k\delta \rangle\Omega_\lambda (x;\alpha)$$
where as usual $\delta = (n-1, n-2,\ldots,1,0).$ 

$\big[$From loc.\ cit., the eigenvalue is
\begin{eqnarray*}
& & 2\alpha^{-1}\left(\alpha n(\lambda')-n(\lambda)+\left((n-1)+\tfrac{1}{2}\alpha\right)|\lambda|\right) \\
&=& 2n(\lambda')-2kn(\lambda)+2k(n-1)|\lambda|+|\lambda| \\
&=& \Sigma\lambda_i (\lambda_i -1-2k(i-1)+2k(n-1)+1) \\
&=& \Sigma\lambda_i(\lambda_i+2k(n-i)).\big]
\end{eqnarray*}

Consider
\begin{eqnarray*}
E_y \int e(-x,y)|x|^{a-p}\Omega_\lambda(x;\alpha)\mathrm{d}\mu(x) &=& \int E_x(e(-x,y))|x|^{a-p}\Omega_\lambda(x;\alpha)\mathrm{d}\mu(x) \\
&=& \sum^n_{i=1} \int x_i D_i (\Delta(x)^{2k} x_i D_i e(-x,y)) |x|^{a-p}\Omega_\lambda (x)\mathrm{d}x.
\end{eqnarray*}

If integration by parts is o.k.\ in these circumstances -- we need to have $D_i(e(-x,y))$ vanishing more strongly than any polynomial in $x_i$ as $x_i\rightarrow \infty$ -- we can replace this by\footnote{See Additional observation 2 at the end of this section.}
$$
 - \sum^n_{i=1} \int x_iD_i(e(-x,y)) D_i (x_i |x|^{a-p}\Omega_\lambda(x)) \Delta(x)^{2k}\mathrm{d}x
$$

and then, integrating by parts again, we obtain
$$
\sum^n_{i=1} \int e(-x,y) D_i(x_i\Delta(x)^{2k}D_i(x_i|x|^{a-p}\Omega_\lambda(x)))\mathrm{d}x
$$
or, if we introduce the operator $E'$ defined by
$$E'f=\Delta(x)^{-2k}\sum^n_{i=1}D_i(x_i\Delta(x)^{2k}D_i(x_i f)),$$
we have
$$E_y \int e(-x,y)|x|^{a-p} \Omega_\lambda (x;\alpha)\mathrm{d}\mu(x) = \int e(-x,y) E'(|x|^{a-p}\Omega_\lambda(x))\mathrm{d}\mu(x).$$

Now
\begin{eqnarray*}
Ef &=& \sum^n_{i=1} x_i(2kD_i(\log \Delta)x_i D_i f+D_i f+x_i D_i^2 f) \\
&=& \sum^n_{i=1}(x_i^2 D_i^2 f+x_i D_i f)+2k\sum_{i\neq j}\frac{x_i^2}{x_i-x_j}D_i f
\end{eqnarray*}
and
\begin{eqnarray*}
E'f &=& \sum^n_{i=1}(D_i(x_i f)+2kx_i (\log \Delta)D_i(x_i f) +x_iD^2_i (x_i f)) \\
&=& nf + \sum^n_{i=1} (x^2_i D^2_i f + 3x_i D_i f) +2k\sum_{i\neq j}\frac{x_i}{x_i-x_j}(f+x_iD_i f) \\
&=& (np+2\sum^n_{i=1}x_iD_i)f+Ef \quad\quad\quad \left(\mbox{since }\sum_{i\neq j}\frac{x_i}{x_i-x_j} = \frac{1}{2}n(n-1)\right)
\end{eqnarray*}
so that
$$E'(|x|^{a-p}\Omega_\lambda (x))=(\langle\mu, \mu+2k\delta\rangle +2|\mu|+np)|x|^{a-p}\Omega_\lambda(x),$$
where $\mu_i = a-p+\lambda_i$. From this it follows that the Laplace transform 
$$\int e(-x,y)|x|^{a-p}\Omega_\lambda(x;\alpha)\mathrm{d}\mu(x)$$
is an eigenfunction of $E_y$ with eigenvalue
$$\langle\mu,\mu+2k\delta\rangle+2|\mu|+np.$$
We want to express this in terms of $\nu$, where
$$\nu_i=-a-\lambda_{n+1-i}=-p-\mu_{n+1-i}.$$
So we have
$$\mu_i=-p-\nu_{n+1-i}$$
and
\begin{eqnarray*}
\langle \nu, \nu+2k\delta\rangle &=& \sum^n_{i=1}\nu_{n+1-i} (\nu_{n+1-i}+2k(i-1)) \\
&=& \sum^n_{i=1}(\mu_i+p)(\mu_i+p-2k(i-1)).
\end{eqnarray*}
Now $p=(n-1)k+1$, so that
$$ 2p-2k(i-1)=2k(n-i)+2$$
and hence
\begin{eqnarray*}
\langle \nu, \nu+2k\delta\rangle &=& \sum^n_{i=1}(\mu^2_i +2\mu_i(k(n-i)+1))+p\sum^n_{i=1}(p-2k(i-1)) \\
&=& \langle\mu, \mu +2k\delta\rangle+2|\mu|+np
\end{eqnarray*}
(since $p-2k(i-1) = k(n+1-2i)+1$).

So finally we have
$$E_y \int e(-x,y)|x|^{a-p}\Omega_\lambda(x;\alpha)\mathrm{d}\mu(x) = \langle \nu, \nu +2k\delta\rangle \int e(-x,y)|x|^{a-p}\Omega_\lambda(x;\alpha)\mathrm{d}\mu(x)$$
which in view of (C$'$) shows that the integral must be a scalar multiple of 
$$|y|^{-a}\Omega_\lambda(y^{-1};\alpha).$$
Finally, by setting $y=1$ we get the scalar constant.
\bigskip

\subsection*{Additional observation 1}

\eqref{conjC} means that for any $F(x,y)$ we have 
$$\int e^{-\mathrm{tr}(x)}F(x,y)|x|^{a-p}\mathrm{d}\mu(x) = |y|^{-a}\int e(-x,y)F(x,1)|x|^{a-p}\mathrm{d}\mu(x)$$
---i.e., the two versions of Laplace transform agree.
\bigskip

\subsection*{Additional observation 2}

When $\alpha = \infty$ (i.e., $k=0$) we have
$$\alpha^{|\lambda|}J_{\lambda}^*(x;\alpha) = \frac{m_{\lambda}(x)}{\lambda!}$$
$$\Gamma_n (a+p;\lambda;\alpha) = (a+\lambda)!$$
and therefore
\begin{eqnarray*}
A_a(x;\infty) &=& \sum_{\lambda}\frac{(-1)^{|\lambda|}m_{\lambda}(x)}{\lambda!(a+\lambda)!} \\
&=& \sum_{\alpha\in\mathbb{N}^n}\frac{(-1)^{|\alpha|}x^{\alpha}}{\alpha!(a+\alpha)!} \\
&=& \prod_{i=1}^n\mbox{}_0F_1(a;-x_i)
\end{eqnarray*}
and likewise
$$A_a(x;y;\infty) = \frac{1}{n!}\sum_{\omega\in S_n}\prod_{i=1}^n\mbox{}_0F_1(a;-x_iy_{\omega (i)}).$$
\newpage

\section*{Fourier transform} 
Let $F=\mathbb{R}, \mathbb{C}$ or $\mathbb{H}$ and let
\begin{eqnarray*}
k &=& \tfrac{1}{2}(F:\mathbb{R})\mbox{  }(=\tfrac{1}{2}, 1 \mbox{ or } 2) \\
p &=& k(n-1)+1 \mbox{  } (=\tfrac{1}{2}(n+1),n \mbox{ or } 2n-1).
\end{eqnarray*}
Let $\Sigma = \Sigma_n(F)$ denote the space of $n\times n$ hermitian matrices with entries in $F$, so that
$${\rm dim}_{\mathbb{R}}\Sigma=n+kn(n-1)=np.$$
If $s=(s_{ij})\in\Sigma$ we have $s_{ii} \in \mathbb{R}$ and $s_{ji} = \bar s_{ij}\mbox{  }(i<j).$ \\
We write $$ s_{ij}=\sum_{\alpha}s_{ij\alpha}e_{\alpha}\quad\quad(i<j)$$
where $(e_\alpha)_{0\leq\alpha\leq2k-1}$ is the standard basis of $F$ over $\mathbb{R}$ (with $e_o=1$). We take the functions $s_{ii}(1\leq i\leq n)$,  $s_{ij\alpha}(i<j, 0\leq\alpha\leq2k-1)$ as coordinate functions on $\Sigma$ and define
$$\mathrm{d}s=\pi^{-kn(n-1)/2}\left(\prod^n_{i=1}\mathrm{d}s_{ii}\right)\left(\prod_{i<j}\prod_\alpha \mathrm{d}s_{ij\alpha}\right)$$
If $x\in M_n(\mathbb{H})$ we define
$${\rm trace}(x)={\rm Re}\left(\sum^n_{i=1}x_{ii}\right)$$
(so the trace is a real number; we have ${\rm trace}(x)=\frac{1}{2}{\rm trace}\, \omega (x)$ where $\omega: M_n(\mathbb{H})\rightarrow M_{2n}(\mathbb{C})$ is the standard embedding, induced by $\omega(a+bi+cj+dk)=\left(\begin{matrix}a+bi&c+di\\-c+di&a-bi\end{matrix}\right)$.)

Then if $s,t\in\Sigma$ we have
\begin{eqnarray*}
{\rm trace}(st) &=& {\rm Re}\sum^n_{i,j=1} s_{ij}t_{ji} = {\rm Re}\sum_{i,j}s_{ij}\bar t_{ij} \\
&=& \sum^n_{i=1}s_{ii}t_{ii}+\sum_{i<j}{\rm Re}(s_{ij}\bar t_{ij} + \bar s_{ij}t_{ij}) \\
&=& \sum^n_{i=1}s_{ii}t_{ii}+2\sum_{i<j}\sum_\alpha s_{ij\alpha}t_{ij\alpha}.
\end{eqnarray*}

The \underline{Fourier transform} of a function $f$ on $\Sigma$ is
$$\hat f(t)=\int_{\Sigma}e^{-i \,\tr (st)}f(s)\mathrm{d}s$$
and we shall have 
$$\hat{\hat{f}}(s)=c_nf(-s)$$
for some constant $c_n$. To evaluate the constant we may take $f(s)=e^{-\frac{1}{2}\tr (s^2)}$, so that
\begin{eqnarray*}
\hat f(t) &=& \int_\Sigma {\rm exp}-\frac{1}{2}\tr (s^2 + 2ist)\mathrm{d}s \\
&=& e^{-\frac{1}{2}\tr t^2}\int_\Sigma {\rm exp}-\frac{1}{2} \tr (s+it)^2 \mathrm{d}s
\end{eqnarray*}
Now (by Cauchy's theorem) this integral is equal to
$$\int_\Sigma {\rm exp}-\frac{1}{2}\tr (s^2)\mathrm{d}s = \left(\prod^n_{i=1}\int_{\mathbb{R}}e^{-s_{ii}^2/2}\mathrm{d}s_{ii}\right)\cdot\left(\prod_{i<j}\prod_\alpha\int_{\mathbb{R}}e^{-s^2_{ij\alpha}}\mathrm{d}s_{ij\alpha}\right)\pi^{-kn(n-1)/2}$$
Since $$\int_{\mathbb{R}}e^{-x^2}\mathrm{d}x = \pi^{1/2},\quad \int_{\mathbb{R}}e^{-x^2/2}\mathrm{d}x = (2\pi)^{1/2}$$
it follows that 
$$ \hat f (t) = (2\pi)^{n/2}f(t)$$
and therefore $\hat{\hat f}(t) = (2\pi)^n f(t)$, i.e., $c_n=(2\pi)^n$:
$$\boxed{\hat{\hat f}(s) = (2\pi)^n f(-s).}$$

Now let $$K=U_n(F)=\{k\in M_n(\mathbb{F}):\;  k\bar k' =1\}$$
and suppose that $f$ is $K$-invariant: $f(ksk^{-1}) = f(s)$.

Since $d(ksk^{-1})=ds$, we have
\begin{eqnarray*}
\hat f(t) &=& \int_\Sigma e^{-i\,\tr (ksk^{-1}t)}f(s)\mathrm{d}s \\
&=& \int_\Sigma e(-is,t)f(s)\mathrm{d}s,
\end{eqnarray*}
where $$e(s,t)=\int_K e^{\tr (ksk^{-1}t)}\mathrm{d}k$$
($\mathrm{d}k$ being normalized Haar measure on the compact group $K$).

We can write this in the form $$\hat f(y)=\int_{\mathbb{R}^n}e(-ix,y)f(x)\mathrm{d}\sigma(x).$$
where $$\mathrm{d}\sigma(x)=c'_n|\Delta(x)|^{2k}\mathrm{d}x$$
$(\Delta(x)=\prod_{i<j}(x_i-x_j) ;\;  \mathrm{d}x=\mathrm{d}x_1\cdots \mathrm{d}x_n)$ and $c_n$ is some constant. So for $k=2,1,\frac{1}{2}$ we have the \underline{Fourier transform} (of suitable functions $f$)
\begin{equation}
\hat f(y) = \int_{\mathbb{R}^n} e(-ix, y) f(x)\mathrm{d}\sigma(x)
\nonumber
\tag{F1}
\end{equation}
and \underline{Fourier reciprocity}
\begin{equation}
\hat{\hat f}(x) = (2\pi)^n f(-x).
\nonumber
\tag{F2}
\end{equation}
Now the \underline{exponential kernel} $e(x,y)$ is defined for all $k$, namely if $C_\lambda(x)=C_\lambda(x,k)$ is the multiple of Jack's symmetric function $J_\lambda(x;k^{-1})$ such that 
$$\sum_{|\lambda|=m} C_\lambda(x)=p_1(x)^m$$
then
$$e(x,y)=\sum_\lambda\frac{C_\lambda(x)C_\lambda(y)}{C_\lambda(1_n)|\lambda|!}.$$
So in particular
$$e(x,1_n)=\sum_{m\geq 0}\frac{p_1(x)^m}{m!}=e^{\sum x_i}$$
We may therefore define the Fourier transform (F1) for all $k$. When $k=2,1$ or $\frac{1}{2}$ the reciprocity formula (F2) will be true, for the reasons given above.
[Another value of $k$ for which it holds is $k=0$. For then (p.\ L2)
$$e(x,y)=\frac{1}{n!}\sum_{\omega \in S_n}e^{\langle x,\omega y\rangle}$$ 
and (F1) gives 
\begin{eqnarray*}
\hat f(y) &=& \frac{1}{n!}\sum_{\omega \in S_n}\int_{\mathbb{R}^n}e^{-i\langle x,\omega y\rangle}f(x)\mathrm{d}x \\
&=& \int_{\mathbb{R}^n} e^{-i\langle x,y\rangle}f(x)\mathrm{d}x
\end{eqnarray*}
by symmetry, i.e., $\hat f(y)$ is the usual Fourier transform on $\mathbb{R}^n$.
So in this case (F2) is valid.]

Coming back to the cases $k=2,1,\frac{1}{2}$ we had ($x,y$ real diagonal matrices)
$$e(x,y)=\int_K e^{\tr (xkyk^{-1})}\mathrm{d}k$$
(double use of the letter $k  !$) so that
$$\frac{\partial}{\partial x_i} e(x,y)=\int_K (kyk^{-1})_i e^{\tr(xkyk^{-1})}\mathrm{d}k$$
and therefore, if we put
\begin{eqnarray*}
D_x &=& \sum^n_{i=1}\frac{\partial}{\partial x_i}, \\
D_x \, e(x,y) &=&\int_K {\rm trace}(kyk^{-1})\cdot e^{\tr (xkyk^{-1})}\mathrm{d}k \\
&=& {\rm trace}(y)\cdot e(x,y)
\end{eqnarray*}
or equivalently
$$\boxed{D_x e(x,y) = p_1(y)e(x,y)}$$
This is probably the fundamental property of the exponential kernel $e(x,y)$  ;  it is true for all values of $\alpha(=k^{-1})$. So it is indicated that the Fourier reciprocity theorem (F2) should be a consequence of this:

\underline{Example} of Fourier transform. Let
$$
f(x)=\left\{\begin{tabular}{ll}
$|x|^{a-p}|1-x|^{b-p}$, & if $x\in[0,1]^n$ \\
$0$, & otherwise
\end{tabular}
\right.
$$

Then 
\begin{eqnarray*}
\hat f(y) &=& \int_{\mathbb{R}^n}e(-ix,y)f(x)\mathrm{d}\sigma(x) \\
&=& \int_{[0,1]^n}e(-ix,y)|x|^{a-p}|1-x|^{b-p}\mathrm{d}\sigma(x) \\
&=& \sum_\lambda \frac{C_\lambda (y)}{|\lambda|!}(-i)^{|\lambda|}\int_{[0,1]^n}\Omega_\lambda(x) |x|^{a-p} |1-x|^{b-p}\mathrm{d}\sigma(x) \\
&=& \frac{\Gamma_n(a)\Gamma_n(b)}{\Gamma_n(a+b)}\sum_\lambda \frac{(a)_\lambda}{(a+b)_\lambda}\cdot \frac{C_\lambda(-iy)}{|\lambda|!}
\end{eqnarray*}
by Kadell's integral, see \eqref{eq_6.19}, p.26.

In other words
$$\hat f(y) = \frac{\Gamma_n(a)\Gamma_n(b)}{\Gamma_n(a+b)}  \mbox{ }_1F_1(a;a+b;-iy)$$
is the Fourier transform.

\medskip

But  to compute $\hat{\hat f}$ we cannot again integrate term by term.

\bigskip

\underline{Laplace transform}
\smallskip
\smallskip

This is
$$(Lf)(y)=\left(\int^\infty_0\right)^n e(-x,y)f(x)\mathrm{d}\sigma(x)$$
so that $Lf(iy)$ is the Fourier transform of $f(x)\chi(x)$ where $\chi$ is the characteristic function of the positive octant in $\mathbb{R}^n$. Hence
$$(2\pi)^nf(x)\chi(x) = \int_{\mathbb{R}^n}e(x,iy)Lf(iy)\mathrm{d}\sigma(y)$$
which is a version of the inverse Laplace transform (for suitable $f$).
\newpage

\section*{Differential equations for hypergeometric functions} 
Introduce the differential operators
$$\delta_r=\sum^n_{i=1}x^r_i D^2_{x_i}+2k\sum_{i\neq j}\frac{x^r_i D_{x_i}}{x_i-x_j}$$ 
$$\varepsilon_r=\sum^n_{i=1}x_i^{r-1}D_{x_i} \quad \quad \quad \quad (r\geq1)$$
where $D_{x_i}=\partial/\partial x_i$. 
For a partition $\lambda$, let
$$ \rho(\lambda)=n(\lambda')-kn(\lambda)$$
and for $\lambda \supset \mu$ Let
$$\rho(\lambda/\mu)=\rho(\lambda)-\rho(\mu).$$
If $\lambda-\mu$ is a single square $(i,j)$, then
$$\rho(\lambda/\mu)=j-1-k(i-1)=\mu_i-k(i-1)$$
and hence
$$(a)_\lambda/(a)_\mu=a+\rho(\lambda/\mu)$$
in this case.

\begin{lemmaB} 
Let
$$\Box_r=\frac{1}{r}\delta_r-k(n-1)\varepsilon_r \quad (r=1,2)$$
Then
\begin{itemize} 
\item[(i)]    $\Box_2\Omega_\lambda \mbox{  }=\mbox{  } \rho(\lambda)\Omega_\lambda \mbox{  },\mbox{  }\varepsilon_2\Omega_\lambda\mbox{  }=\mbox{  }|\lambda|\Omega_\lambda$  
\item[(ii)]    $\varepsilon_1\Omega_\lambda\mbox{  }=\mbox{  }\sum_\mu\binom{\lambda}{\mu}\Omega_\mu$ 
\item[(iii)]   $\Box_1\Omega_\lambda\mbox{  }=\mbox{  }\sum_\mu \rho(\lambda/\mu)\binom{\lambda}{\mu}\Omega_\mu$ 
\item[(iv)]   $[\Box_1,\Box_2]\mbox{  }=\mbox{  }\sum_\mu \rho(\lambda/\mu)^2\binom{\lambda}{\mu}\Omega_\mu$ 
\end{itemize}
The sums in {\rm  (ii)--(iv)} are over $\mu\subset\lambda$ such that $|\lambda-\mu|=1$. Here $\binom{\lambda}{\mu}=\varepsilon(J_{\lambda/\mu})$,   $\varepsilon$ the specialization defined by $\varepsilon(p_r)=\delta_{1r}$ (so   $\varepsilon(E(t))=\varepsilon(H(t))=e^t.$) 
\end{lemmaB}

\begin{proof} 
\mbox{}
\begin{itemize}
\item[(i)]     is a restatement of the fact that $\Omega_\lambda$ is an eigenfunction of the Laplace-Beltrami operator. 
\item[(ii)]    is proved elsewhere.
\item[(iii)]   follows from (i), (ii) and the fact that $\Box_1 = [\varepsilon_1 , \Box_2]$.
\item[(iv)]   follows from (i), (iii).
\end{itemize}
\end{proof}

\underline{Remark}.  More generally, we have
$$\sum_\mu \binom{\lambda}{\mu}(\rho(\lambda/\mu))^r \Omega_\mu = (-1)^r(({\rm ad}\, \Box_2)^r\varepsilon_1)\Omega_\lambda$$
for all $r\geq0$.

\begin{lemmaB}
\mbox{}
\begin{itemize}
\item[(i)]    $\varepsilon_1 e^{p_1}=ne^{p_1},\mbox{    }\varepsilon_2 e^{p_1} = p_1 e^{p_1}.$ 
\item[(ii)]    $\Box_1 e^{p_1}=p_1e^{p_1},\mbox{    }\Box_2 e^{p_1} = \frac{1}{2}p_2 e^{p_1}.$ 
\item[(iii)]   $[\Box_1,\Box_2]e^{p_1} = (p_1(k(n-1)+1)+p_2)e^{p_1}.$
\end{itemize}
\end{lemmaB}

\begin{proof} (i) is clear since $D_{x_i}e^{p_1}=e^{p_1}(1\leq i \leq n)$ \\
(ii) we have
\begin{eqnarray*}
\Box_1 &=& \sum^n_{i=1}x_i D_{x_i}^2 + 2k\sum_{i\neq j} \frac{x_i}{x_i - x_j}D_i - k(n-1)\sum^n_{i=1}D_i  \\
&=& \sum^n_{i=1} x_i D^2_{x_i}+ k\sum_{i\neq j}\frac{x_i+x_j}{x_i-x_j}D_i
\end{eqnarray*}
so that $$ \Box_1 e^{p_1} = \left(p_1 + k\sum_{i\neq j}\frac{x_i+x_j}{x_i-x_j}\right)e^{p_1}=p_1 e^{p_1} .$$
Likewise, 
\begin{eqnarray*}
\Box_2 &=& \frac{1}{2}\sum x_i^2D_{x_i}^2 + k\sum_{i\neq j}\frac{x_i^2}{x_i-x_j}D_{x_i}-k(n-1)\sum x_iD_i \\
&=& \frac{1}{2}\sum x_i^2 D_{x_i}^2 + k\sum_{i\neq j}\frac{x_i x_j}{x_i-x_j}D_{x_i}
\end{eqnarray*}
so that
$$ \Box_2 e^{p_1}=\left(\frac{1}{2}p_2+k\sum_{i\neq j} \frac{x_ix_j}{x_i-x_j}\right)e^{p_1}=\frac{1}{2}p_2e^{p_1}.$$
(iii) We have
$$[\Box_1,\Box_2]e^{p_1}=\Box_1\left(\frac{1}{2}p_2e^{p_1}\right)-\Box_2(p_1 e^{p_1}).$$
Now $$ D_{x_i}^2(p_2e^{p_1}) = (2 + 4x_i + p_2) e^{p_1}$$
so that
\begin{eqnarray*}
\Box_1(\frac{1}{2}p_2e^{p_1}) &=&\frac{1}{2}\sum(2x_i+4x_i^2+x_i p_2)e^{p_1}+\frac{1}{2}k\sum_{i\neq j}\frac{x_i+x_j}{x_i-x_j}(2x_i+p_2)e^{p_1} \\
&=& \left(p_1+2p_2+\frac{1}{2}p_1 p_2 +k\sum_{i\neq j}\frac{x_i^2+x_ix_j}{x_i-x_j}\right)e^{p_1} \\
&=& \left(p_1(1+k(n-1))+2p_2+\frac{1}{2}p_1 p_2\right)e^{p_1}.
\end{eqnarray*}
Next, we have
\begin{eqnarray*}
\Box_2 (p_1e^{p_1}) &=& \frac{1}{2}\sum x_i^2(2+p_1)e^{p_1}+k\sum_{i\neq j}\frac{x_ix_j(1+p_1)}{x_i-x_j}e^{p_1} \\
&=& \left(p_2+ \frac{1}{2}p_1 p_2\right)e^{p_1}
\end{eqnarray*}
so that finally
$$[\Box_1,\Box_2]e^{p_1}=(p_1(1+k(n-1))+p_2)e^{p_1}. $$
\end{proof}

\begin{lemmaB} 
\mbox{}
\begin{itemize}
\item[(i)]    ${\displaystyle \sum_\lambda} \binom{\lambda}{\mu} J^*_\lambda (1)=nkJ^*_\mu (1)$ 
\item[(ii)]   ${\displaystyle \sum_\lambda} \binom{\lambda}{\mu}\rho(\lambda/\mu)J^*_\lambda (1) = |\mu|kJ^*_\mu (1)$ 
\item[(iii)]   ${\displaystyle \sum_\lambda} \binom{\lambda}{\mu} \rho(\lambda/\mu)^2 J^*_\lambda (1) = ((1+(n-1)k)|\mu|+2\rho(\mu))kJ^*_\mu(1)$, 
\end{itemize}
summed in each case over $\lambda\supset\mu$ such that $|\lambda-\mu|=1$.
\end{lemmaB}

\begin{proof} Operate on $e^{p_1}=\sum_\lambda \alpha^{|\lambda|}J^*_\lambda$ with (i) $\varepsilon_1$, (ii) $\Box_1$, (iii) $(\Box_1 , \Box_2)$. \\
(i) We obtain
\begin{eqnarray*}
n\sum_\mu \alpha^{|\mu|}J^*_\mu &=& \sum_\lambda \alpha^{|\lambda|} \varepsilon_1(J^*_\lambda) \\
&=& \sum_\lambda \alpha^{|\lambda|}J^*_\lambda (1)\sum_{\mu\subset\lambda}\binom{\lambda}{\mu}J^*_\mu/J^*_\mu (1)
\end{eqnarray*}
by Lemma 1 (ii). Now equate coefficients of $J^*_\mu$:
$$nJ^*_\mu (1) = \alpha\sum_\lambda\binom{\lambda}{\mu}J^*_\lambda (1).$$
(ii) Likewise, since $$\Box_1e^{p_1}=p_1e^{p_1}=\varepsilon_2e^{p_1}=\sum_\mu|\mu|\alpha^{|\mu|}J^*_\mu ,$$
we have
\begin{eqnarray*}
\sum_\mu |\mu|\alpha^{|\mu|}J^*_\mu &=& \sum_\lambda \alpha^{|\lambda|}\Box_1 (J^*_\lambda) \\
&=&\sum_\lambda \alpha^{|\lambda|}J^*_\lambda (1)\sum_{\mu\subset\lambda}\binom{\lambda}{\mu}\rho(\lambda/\mu)J^*_\mu/J^*_\mu (1)
\end{eqnarray*}
by Lemma 1(iii). So we obtain $$|\mu|J^*_\mu (1) = \alpha \sum_\lambda \binom{\lambda}{\mu}\rho(\lambda/\mu)J^*_\lambda (1).$$
(iii) Finally, we have from Lemma 2.
$$[\Box_1 , \Box_2]e^{p_1}=(1+k(n-1))\varepsilon_2 e^{p_1}+2\Box_2 e^{p_1}$$
and therefore
$$
\sum_\mu ((1+k(n-1))|\mu|+2\rho(\mu))\alpha^{|\mu|}J^*_\mu 
= \sum_\lambda \alpha^{|\lambda|} J^*_\lambda (1)\sum_{\mu\subset\lambda}\binom{\lambda}{\mu}\rho(\lambda/\mu)^2J^*_\mu/J^*_\mu (1)
$$
using Lemma 1 (iv); hence
$$((1+k(n-1))|\mu|+2\rho(\mu))J^*_\mu (1) = \alpha \sum_{\lambda\supset\mu}\binom{\lambda}{\mu}\rho(\lambda/\mu)^2 J^*_\lambda (1).$$
\end{proof}

Now consider
$$ _2 F_1 (a,b;c;x)=\sum_\lambda \frac{(a)_\lambda (b)_\lambda}{(c)_\lambda}\alpha^{|\lambda|}J^*_\lambda.$$
Since
\begin{eqnarray*}
(\Box_1 + c\varepsilon_1)\Omega_\lambda &=& \sum_\mu \binom{\lambda}{\mu}(\rho(\lambda/\mu)+c)\Omega_\mu \\
&=& \sum_\mu \binom{\lambda}{\mu} \frac{(c)_\lambda}{(c)_\mu}\Omega_\mu
\end{eqnarray*}
it follows that
\begin{eqnarray*}
(\Box_1 +c\varepsilon_1)\cdot \mbox{}_2F_1 &=& \sum_{\substack{\lambda,\mu\\ \lambda\supset\mu}} \frac{(a_\lambda)(b_\lambda)}{(c)_\mu} \binom{\lambda}{\mu}\alpha^{|\lambda|}J^*_\lambda(1)\Omega_\mu\\
&=& \sum_\mu \frac{(a)_\mu (b)_\mu}{(c)_\mu}\Omega_\mu \cdot \sum_{\lambda\supset\mu}(a+\rho(\lambda/\mu))(b+\rho(\lambda/\mu))\binom{\lambda}{\mu}\alpha^{|\lambda|}J^*_\lambda(1).
\end{eqnarray*}
Now from Lemma 3 we have
$$
\sum_{\lambda\supset\mu}(a+\rho(\lambda/\mu))(b+\rho(\lambda/\mu))\binom{\lambda}{\mu}J^*_\lambda (1) 
= k(abn+(a+b)|\mu|+((n-1)k+1)|\mu|+2\rho(\mu))J^*_\mu (1)
$$
and therefore
\begin{eqnarray*}
(\Box_1 + c\varepsilon_1)\cdot\mbox{}_2 F_1 &=& \sum_\mu \frac{(a)_\mu (b)_\mu}{(c)_\mu}\alpha^{|\mu|} (abn+(a+b+1+(n-1)k)|\mu|+2\rho(\mu))J^*_\mu \\
&=& (abn+(a+b+1+(n-1)k)\varepsilon_2 + 2\Box_2)\,\mbox{}_2F_1 \\
&=& (abn+(a+b+1-(n-1)k)\varepsilon_2+\delta_2)\,\mbox{}_2 F_1.
\end{eqnarray*}
Hence the differential equation satisfied by $f=\mbox{}_2F_1$  is
$$_2 \Phi_1 (f)=0$$
where
\begin{eqnarray*}
_2 \Phi_1 &=& \delta_2 - \delta_1 + (a+b+1-(n-1)k)\varepsilon_2 -(c-(n-1)k)\varepsilon_1 + abn \\
&=& \sum^n_{i=1}\left\{(x_i^2-x_i)D_{x_i}^2+(a+b+1)x_iD_i-cD_i+ab\right\} + (n-1)k\sum^n_{i=1}(1-x_i)D_{x_i}+2k\sum_{i\neq j} \frac{x^2_i -x_i}{x_i-x_j}D_{x_i}
\end{eqnarray*}
i.e.,
$$_2\Phi_1 =\sum^n_{i=1}\left\{(x_i D_{x_i}+a)(x_iD_{x_i}+b)-D_{x_i}(x_iD_{x_i}+c-1)\right\} + k\sum_{i\neq j}(x_i -1)\frac{x_i+x_j}{x_i -x_j} D_{x_i}$$
The ``error term'' in the 2nd line is a derivation.

Likewise, for $_1F_1 (a;c;x)$ the corresponding operator is
\begin{eqnarray*}
_1\Phi_1 &=& \delta_1 + (c-(n-1)k)\varepsilon_1 - \varepsilon_2 - na \\
&=& \sum^n_{i=1}\left\{D_{x_i}(x_i D_{x_i}+c-1)-(x_i D_{x_i}+a)\right\}+2k\sum_{i\neq j}\frac{x_i}{x_i - x_j}D_{x_i}-(n-1)k\sum D_i \\
&=& \sum^n_{i=1}\left\{D_{x_i}(x_i D_{x_i}+c-1)-(x_i D_{x_i}+a)\right\}+k\sum_{i\neq j}\frac{x_i+x_j}{x_i-x_j}D_{x_i}.
\end{eqnarray*}
For $_0F_1(c;x)$ it is
\begin{eqnarray*}
_0\Phi_1 &=& \delta_1+(c-(n-1)k)\varepsilon_1-n \\
&=& \sum^n_{i=1}(D_{x_i}(x_i D_{x_i}+c-1)-1)+k\sum_{i\neq j}\frac{x_i+x_j}{x_i-x_j}D_{x_i}.
\end{eqnarray*}

For the hypergeometric kernels $_pF_q(x,y)$ the differential operators have been calculated\footnote{Editorial note: Constantine and Muirhead calculate the differential operators only for ${}_2F_1$ and
degenerate cases, not for general ${}_pF_q$.
For ${}_3F_2$ this was done by Fujikoshi \cite{Fujikoshi75}.} by Constantine \& Muirhead \cite{ConstantineMuirhead72} in the case $k=\frac{1}{2}$, presumably they are as follows for arbitrary $k$:--

\bigskip
\indent\indent
\begin{tabular}{lll}
$_0F_0(x,y)$ &\phantom{XXX} &$\delta_{1,x}-\varepsilon_{3,y}-k(n-1)p_1(y)$ \\
$_1F_0(x,y)$ & &$\delta_{1,x}-\delta_{3,y}-(a+1-k(n-1))\varepsilon_{3,y}-ak(n-1)p_1(y)$ \\
$_0F_1(x,y)$ & &$\delta_{1,x}+(c-k(n-1))\varepsilon_{1,x}-p_1(y)$ \\
$_1F_1(x,y)$ & &$\delta_{1,x}+(c-k(n-1))\varepsilon_{1,x}-\varepsilon_{3,y}-ap_1(y)$ \\
$_2F_1(x,y)$ & &$\delta_{1,x}+(c-k(n-1))\varepsilon_{1,x}-(a+b)\varepsilon_{3,y}-\delta_{3,y}-abp_1(y)$,
\end{tabular}

\bigskip
\noindent
where 
\begin{eqnarray*}
& & \delta_{3,x}=\sum x_i^3 D^3_{x_i}+2k\sum_{i\neq j}\frac{x^3_i}{x_i-x_j}D_{x_i} \\
& & \varepsilon_{3,x}=\sum x^2_i D_{x_i}.
\end{eqnarray*}
\newpage

\section{Jacobi polynomials}
(For the cases $\alpha = 2,1$ see James \& Constantine  \cite{JamesConstantine74}.)

Let $a,b>0$ \& with our usual notation define

\begin{equation} 
u_{a.b}(x) = u(x) = |x|^a|1-x|^b\Delta(x)^{2k}
\label{eq_9.1}
\end{equation}
and a scalar product
\begin{equation} 
\langle f,g \rangle_{a,b} = \left(\int^1_0\right)^n f(x)g(x)u(x)\mathrm{d}x.
\label{eq_9.2}
\end{equation}
Let $E=E_{a,b}$ be the differential operator defined by
\begin{equation} 
Ef=-u(x)^{-1}\sum^n_{i=1}D_i(x_i(1-x_i)u(x)D_if)
\label{eq_9.3}
\end{equation}
where $D_i=\partial/\partial x_i$.

\begin{flalign}
\label{eq_9.4}
\mbox{$E$ is self-adjoint for the scalar product \eqref{eq_9.2}, i.e.,}&& 
\end{flalign}
$$\langle Ef,g \rangle = \langle f, Eg \rangle.$$

\begin{proof} By definition we have
\begin{eqnarray*}
\langle Ef,g \rangle &=& -\left( \int^1_0 \right)^n \sum^n_{i=1}D_i(x_i(1-x_i)u(x)D_if(x))g(x)\mathrm{d}x \\
&=& + \left( \int^1_0 \right)^n \sum^n_{i=1}x_i(1-x_i)u(x)D_if(x)D_ig(x)\mathrm{d}x
\end{eqnarray*}
on integrating by parts. Since this expression is symmetrical in $f$ and $g$, \eqref{eq_9.4} is proved.
\end{proof}

We need to calculate $E$ more explicitly. For this purpose we introduce the notation
\begin{equation} 
U_{r,s} = \frac{1}{s!}\sum^n_{i=1} x_i^rD_i^s, \quad V_r = \sum_{i\neq j}\frac{x_iD_i}{x_i-x_j}.
\label{eq_9.5}
\end{equation}
Then
\begin{equation} 
E_{a,b} = (a+1)U_{0,1} + (a+b+2)U_{1,1} + 2(U_{2,2} - U_{2,2}) + 2k(V_2-V_1).
\label{eq_9.6}
\end{equation}

\begin{proof} From the definition in \eqref{eq_9.3} we have
\begin{equation*}
Ef = -\sum^n_{i=1}x_i(1-x_i)D_i^2f-\sum(1-2x_i)D_if - \sum_{i=1}^nx_i(1-x_i)D_i(\log u(x))D_if
\end{equation*}
and
\begin{equation*}
D_i(\log u(x)) = \frac{a}{x_i} - \frac{b}{1-x_i} + 2k \sum_{j\neq i}\frac{1}{x_i-x_j}
\end{equation*}
\eqref{eq_9.6} follows easily from these.
\end{proof}

Recall that if
\begin{eqnarray*}
\Box_2 &=& U_{22}+kV_2-k(n-1)U_{11} \\
\Box_1 &=& 2(U_{12}+kV_{1})-k(n-1)U_{01}
\end{eqnarray*}
then
\begin{eqnarray}
\nonumber
\Box_2\Omega_{\lambda} &=& (n(\lambda')-kn(\lambda))\Omega_{\lambda} \\
 &=& \rho(\lambda;\alpha)\Omega_{\lambda} \quad \mbox{say}
\label{eq_9.7}
\end{eqnarray}
and
\begin{equation}
\left\{
\begin{array}{lll}
\displaystyle
\Box_1\Omega_{\lambda} &=&  \displaystyle\sum_{\mu\subset\lambda}\binom{\lambda}{\mu}\rho(\lambda/\mu;\alpha)\Omega_{\mu} \\
\\
U_{01}\Omega_{\lambda} &=&\displaystyle \sum_{\mu\subset\lambda}\binom{\lambda}{\mu}\Omega_{\mu}
\end{array}
\right.
\label{eq_9.8}
\end{equation}
summed in both cases over $\mu\subset\lambda$ such that $|\lambda-\mu|=1$.

\bigskip

From \eqref{eq_9.6} we can rewrite $E_{a,b}$ in the form
\begin{equation}
E_{a,b} = -(a+p)U_{01} + (a+b+2p)U_{11}-\Box_1 + 2\Box_2
\label{eq_9.9}
\end{equation}
where as usual $p=(n-1)k+1$.

On occasion it will be convenient to change the parameters.\footnote{(9.10) in the manuscript reads: ``$A \; \alpha\!\!\!/=a+p,\;\,\, C\; \gamma\!\!\!/=a+b+2p$.'' The $\alpha$ and $\gamma$ are crossed over and replaced by $A$ and $C$, but these changes are not fully propagated in the remainder of the section. We typeset everything with $A$ and $C$ as we believe it to be correct and refer the reader to the manuscript for details should there be confusion with these parameters.}
\begin{equation}
A=a+p,\quad C=a+b+2p
\label{eq_9.10}
\end{equation}
(according with the usage of James \& Constantine) in this notation \eqref{eq_9.9} takes the form
\begin{equation}
E_{a,b} = 2\Box_2 - \Box_1 + CU_{11} - AU_{01}.
\label{eq_9.9p}
\tag{\ref{eq_9.9}$'$}
\end{equation} 

From \eqref{eq_9.7}, \eqref{eq_9.8}, \eqref{eq_9.9p} we have
\begin{equation}
E\Omega_{\lambda} = (C|\lambda| + 2\rho(\lambda))\Omega_{\lambda} - \sum_{\substack{\mu\subset\lambda\\|\lambda-\mu|=1}}(A+\rho(\lambda/\mu))\binom{\lambda}{\mu}\Omega_{\mu}
\label{eq_9.11}
\end{equation}

\underline{Remark}. By comparison, the hypergeometric differential operator is
\begin{equation*}
\Phi_{a,b;c} = 2\Box_2 - \Box_1 + (a+b+p)U_{11} - cU_{10} + abn.
\end{equation*}
So it is not very different from $E_{a,b}$.

We now define \underline{Jacobi polynomials} $G_{\lambda}^{(a,b)}(x;\alpha)$---at present only up to a scalar factor---to be of the form
\begin{equation}
G_{\lambda}^{(a,b)} = \sum_{\mu\subset\lambda}u_{\lambda\mu}\Omega_{\mu}
\label{eq_9.12}
\end{equation}
and eigenfunctions of $E_{a,b}$, 
\begin{equation}
E_{a,b}G_{\lambda}^{(a,b)} = ((a+b+2p)|\lambda|+2\rho(\lambda))G_{\lambda}^{(a,b)}.
\label{eq_9.13}
\end{equation}

The usual argument will show that
\begin{equation}
\langle G_{\lambda},G_{\mu}\rangle_{a,b} = 0 \quad \mbox{if} \quad \lambda\neq\mu
\label{eq_9.14}
\end{equation}
i.e., they are pairwise orthogonal for the scalar product \eqref{eq_9.2}.

We write $G_{\lambda}$ in the form
\begin{equation}
\widetilde{G}_{\lambda}^{(a,b)} = \sum_{\mu\subset\lambda}\frac{(-1)^{|\mu|}}{(a+p)_{\mu}}c_{\lambda/\mu}(a,b)\Omega_{\mu}
\label{eq_9.15}
\end{equation}

Then \eqref{eq_9.11}, \eqref{eq_9.13} give a recursion formula for the coefficients $c_{\lambda/\mu}$. For by operating with $E_{a,b}$ on either side of \eqref{eq_9.15} we obtain, in the notation \eqref{eq_9.10}
\begin{multline*}
(C^{|\lambda|}+2\rho(\lambda))\sum_{\mu\subset\lambda}\frac{(-1)^{|\mu|}}{(A)_{\mu}}c_{\lambda/\mu}\Omega_{\mu} \\
= \sum_{\mu\subset\lambda}\frac{(-1)^{|\mu|}}{(A)_{\mu}}c_{\lambda/\mu}(C|\mu|+2\rho(\mu))\Omega_{\mu} 
 -  \sum_{\substack{\nu\subset\mu\subset\lambda\\|\mu-\nu|=1}}\frac{(-1)^{|\mu|}}{(A)_{\mu}}c_{\lambda/\mu}(A + \rho(\mu/\nu))\binom{\mu}{\nu}\Omega_{\nu}
\end{multline*}
and therefore, equating coefficients of $\Omega_{\mu}$,
\begin{equation*}
(C|\lambda - \mu| + 2\rho(\lambda/\mu))c_{\lambda/\mu} =  \sum_{\substack{\nu\\\mu\subset\nu\subset\lambda\\|\nu-\mu|=1}}(A + \rho(\nu/\mu))\frac{(A)_{\mu}}{(A)_{\nu}}\binom{\nu}{\mu}c_{\lambda/\nu}
\end{equation*}

One checks that $(A)_{\nu}/(A)_{\mu} = A + \rho(\nu/\mu)$, so that we obtain the \underline{recursion} formula
\begin{equation}
(C|\lambda - \mu| + 2\rho(\lambda/\mu))c_{\lambda/\mu} = \sum_{\nu}\binom{\nu}{\mu}c_{\lambda/\nu}\quad\quad (\lambda\neq\mu)
\label{eq_9.16}
\end{equation}
summed over $\nu$ such that $\lambda\supset\nu\supset\mu$ and $|\nu-\mu|=1$.

This shows that $c_{\lambda/\mu}$ is a rational function of $C$ (and $k$): it does not depend on $A$. Normalize it by $c_{\lambda/\lambda} = 1$. The polynomial $\widetilde{G}_{\lambda}^{(a,b)}(1-x)$ must be proportional to $\widetilde{G}_{\lambda}^{(b,a)}(x)$.

 We have
\begin{equation*}
\widetilde{G}_{\lambda}^{(a,b)}(1-x) = \sum_{\mu\subset\lambda}(-1)^{|\mu|}\frac{c_{\lambda/\mu}}{(a+p)_{\mu}}\Omega_{\mu}(1-x)
\end{equation*}
in which the coefficient of $\Omega_{\lambda}(x)$ is $1/(a+p)_{\lambda}$. On the other hand, the coefficient of $\Omega_{\lambda}(x)$ in $\widetilde{G}_{\lambda}^{(b,a)}(x)$ is clearly $(-1)^{|\lambda|}/(b+p)_{\lambda}$, and it follows that
\begin{equation}
(a+p)_{\lambda}\,\widetilde{G}_{\lambda}^{(a,b)}(1-x)=(-1)^{|\lambda|}(b+p)_{\lambda}\,\widetilde{G}_{\lambda}^{(b,a)}(x)
\label{eq_9.17}
\end{equation}
so we define
\begin{equation}
G_{\lambda}^{(a,b)}(x) = (a+p)_{\lambda}\,\widetilde{G}^{(a,b)}_{\lambda}(x)
\label{eq_9.18}
\end{equation}
and then we have symmetry:
\begin{equation}
G_{\lambda}^{(a,b)}(1-x) = (-1)^{|\lambda|} \, G_{\lambda}^{(b,a)}(x).
\label{eq_9.19}
\end{equation}

$\big[$One should probably normalize further:\footnote{See Additional observation at the end of this section.} as at present defined, the constant term in $G_{\lambda}^{(a,b)}(x)$ is $$(a+p)_{\lambda}\, c_{\lambda/0}$$ whereas in the case $n=1$ it is $((\lambda)=r)\binom{a+r}{r}=\frac{(a+1)_r}{r!}.\big]$

Notice that \eqref{eq_9.17} implies the recurrence formula
\begin{equation}
\frac{(b+p)_{\lambda}}{(b+p)_{\nu}}c_{\lambda/\nu} = \sum_{\mu}(-1)^{|\lambda|-|\mu|}\binom{\mu}{\nu}\frac{(a+p)_{\lambda}}{(a+p)_{\mu}}c_{\lambda/\mu}
\label{eq_9.20}
\end{equation}
(and likewise with $a,b$ interchanged ; the $c$'s are functions of $a+b$).

\bigskip

\noindent\underline{Special cases}
\nopagebreak
\smallskip

\noindent (1) $\lambda = (r)$. If $\mu = (s), s\leq r$, then the recurrence \eqref{eq_9.16} gives
\begin{equation*}
(C(r-s) + r(r-1)- s(s-1))c_{r/s} = \binom{s+1}{s}c_{r/(s+1)}
\end{equation*}
i.e.,
\begin{equation*}
(r-s)(C+r+s-1)c_{r/s} = (s+1)c_{r/(s+1)}
\end{equation*}
or
\begin{equation*}
c_{r/s} = \frac{s+1}{(r-s)(C+r+s-1)}c_{r/(s+1)}
\end{equation*}
giving
\begin{equation*}
c_{r/s} = \binom{r}{r-s}\prod_{i=1}^{r-s}\frac{1}{C+r+s+i-2}
\end{equation*}
and hence
\begin{equation}
G_{(r)}^{(a,b)}(x) = \sum^r_{s=0}\frac{(A)_r}{(A)_s}\binom{r}{r-s}\frac{1}{(C+r+s-1)_{r-s}}\Omega_{(s)}
\label{eq_9.21}
\end{equation}

From \eqref{eq_9.24} below it follows that
\begin{eqnarray*}
c_{(1^r)/(1^s)} &=& -\alpha^{r-s}c_{(r)/(s)}(-\alpha C;\alpha^{-1}) \\
&=& \binom{r}{s}\prod_{i=1}^{r-s}\frac{-\alpha}{-\alpha C+ (r+s+i-2)} \\
&=& \binom{r}{s} \prod_{i=1}^{r-s} \frac{1}{C-k(r+s+i-2)}
\end{eqnarray*}
giving
\begin{equation}
G^{(a,b)}_{(1^r)}(x;\alpha) = \sum^s_{r=0}(-1)^s\frac{(A)_{(1^r)}}{(A)_{(1^s)}}\binom{r}{s}\prod^{r-s}_{i=1}(C-k(r+s+i-2)^{-1}\Omega_{(1^s)}
\label{eq_9.21p}
\tag{\ref{eq_9.21}$'$}
\end{equation}

\underline{Duality}
\smallskip

We have
\begin{equation*}
G^{(a,b)}_{\lambda}(x;\alpha) = \sum_{\mu\subset\lambda}(-1)^{|\mu|}\frac{(A;\alpha)_{\lambda}}{(A;\alpha)_\mu}c_{\lambda/\mu}(C;\alpha)\Omega_{\mu}(x;\alpha)
\end{equation*}
and
\begin{eqnarray*}
\Omega_{\mu}(x;\alpha) &=& J_{\mu}(x;\alpha)/J_{\mu}(1_n;\alpha) \\
&=& J_{\mu}(x;\alpha)/\alpha^{|\mu|}(n\alpha^{-1};\alpha)_{\mu}
\end{eqnarray*}
so that
\begin{equation}
G_{\lambda}^{(a,b)}(x;\alpha) = \sum_{\mu\subset\lambda}(-1)^{|\mu|}\frac{(A;\alpha)_{\lambda}}{(A;\alpha)_{\mu}}\cdot \frac{c_{\lambda/\mu}(C;\alpha)}{(n\alpha^{-1};\alpha)_{\mu}}\alpha^{-|\mu|}J_{\mu}(x;\alpha)
\label{eq_9.22}
\end{equation}
and therefore
\begin{equation}
\omega_{\alpha}G_{\lambda}^{(a,b)}(x;\alpha) = \sum_{\mu\subset\lambda}(-1)^{|\mu|}\frac{(A;\alpha)_{\lambda}}{(A;\alpha)_{\mu}}\cdot\frac{c_{\lambda/\mu}(C;\alpha)}{(n\alpha^{-1};\alpha)_{\mu}}J_{\mu'}(x,\alpha^{-1}).
\label{eq_9.23}
\end{equation}
We must therefore express the coefficients in terms of $\lambda'$, $\mu'$.

First of all, we have
\begin{eqnarray*}
(A;\alpha)_{\mu} &=& \prod_{s\in \mu}(A+a'(s) - \alpha^{-1}l'(s)) \\
&=& (-\alpha)^{-|\mu|}\prod_{s\in\mu'}(-A\alpha+a'(s)-\alpha l'(s)) \\
&=& (-\alpha)^{-|\mu|}(-A\alpha;\alpha^{-1})_{\mu'}
\end{eqnarray*}
and likewise
\begin{eqnarray*}
(A;\alpha)_{\lambda} &=& (-\alpha)^{-|\lambda|}(-A\alpha;\alpha^{-1})_{\lambda'}, \\
(n\alpha^{-1};\alpha)_{\mu} &=& (-\alpha)^{-|\mu|}(-n;\alpha^{-1})_{\mu'}.
\end{eqnarray*}

Next, consider $c_{\lambda/\mu} = c_{\lambda/\mu}(C;\alpha)$, which satisfies the recursion
$$ (C|\lambda-\mu|+2\rho(\lambda/\mu;\alpha))c_{\lambda/\mu}(C,\alpha) = \sum_{\nu}\binom{\nu}{\mu}c_{\lambda/\nu}(C,\alpha)$$
summed over $\nu$ such that $\lambda\supset\nu\supset\mu$ and $|\nu-\mu|=1$. Here we have
\begin{eqnarray*}
\rho(\lambda;\alpha) &=& n(\lambda')-\alpha^{-1}n(\lambda) \\
&=& -\alpha^{-1}(n(\lambda)-\alpha n(\lambda')) \\
&=& -\alpha^{-1}\rho(\lambda';\alpha^{-1})
\end{eqnarray*}
so that
$$ C|\lambda-\mu|+2\rho(\lambda/\mu;\alpha) = -\alpha^{-1}(-\alpha C|\lambda'-\mu'|+2\rho(\lambda'/ \mu';\alpha^{-1})).$$
Since moreover $\binom{\nu}{\mu} = \binom{\nu'}{\mu'}$ it follows that 
$$-\alpha^{-1}(-\alpha C|\lambda-\mu|+2\rho(\lambda'/\mu',\alpha^{-1}))c_{\lambda/\mu}(C;\alpha) = \sum_{\nu}\binom{\nu'}{\mu'}c_{\lambda/\nu}(C;\alpha).$$

By comparing this with
$$(-\alpha C|\lambda-\mu|+2\rho(\lambda'/\mu';\alpha^{-1}))c_{\lambda'/\mu'}(-\alpha C,\alpha^{-1}) = \sum_{\nu}\binom{\nu'}{\mu'}c_{\lambda'/\nu'}(-\alpha C;\alpha^{-1})$$
we see that
$$(-\alpha)^{-|\lambda|+|\mu|}c_{\lambda/\mu}(C;\alpha) = c_{\lambda'/\mu'}(-\alpha C;\alpha^{-1})$$
i.e., that
\begin{equation}
c_{\lambda/\mu}(C;\alpha) = (-\alpha)^{|\lambda-\mu|}c_{\lambda'/\mu'}(-\alpha C;\alpha^{-1}).
\label{eq_9.24}
\end{equation}

From these calculations we have $$\frac{(A;\alpha)_{\lambda}}{(A;\alpha)_{\mu}}\cdot\frac{c_{\lambda/\mu}(C;\alpha)}{(n\alpha^{-1};\alpha)_{\mu}} = \frac{(-A\alpha ;\alpha^{-1})_{\lambda'}}{(-A\alpha ;\alpha^{-1})_{\mu'}}\cdot \frac{c_{\lambda'/\mu'}(-C\alpha;\alpha^{-1})}{(-\alpha)^{-|\mu|}(-n;\alpha^{-1})}$$
and therefore from \eqref{eq_9.23}
\begin{eqnarray*}
\omega_{\alpha}G_{\lambda}^{(a,b)}(x;\alpha) 
& = &\sum_{\mu\subset\lambda}(-1)^{|\mu'|} \frac{(-A\alpha ;\alpha^{-1})_{\lambda'}}{(-A\alpha ;\alpha^{-1})_{\mu'}}\cdot 
\frac{c_{\lambda'/\mu'}(-C\alpha;\alpha^{-1})}{(-n;\alpha^{-1})_{\mu'}}\alpha^{|\mu|}J_{\mu'}(-x,\alpha^{-1}) 
\\
&=& G_{\lambda'}^{(a',b')}(-x;\alpha^{-1}),
\end{eqnarray*}
where
$$A'=-A\alpha,\quad C'=-C\alpha,\quad n'\alpha = -n$$
and therefore
$$p' = (n'-1)\alpha+1 = -n-\alpha+1 = -\alpha(\alpha^{-1}(n-1)+1)$$
i.e.,
$$p' = -\alpha p$$
and therefore
$$a' = -\alpha a,\quad b' = -\alpha b.$$
So the duality relation is finally

\begin{equation}
\omega_{\alpha}G_{\lambda,n}^{(a,b)}(x;\alpha) = G_{\lambda',n'}^{(a',b')}(-\alpha,\alpha^{-1}),
\label{eq_9.25}
\end{equation}
where
\begin{equation}
(a',b',n') = (-\alpha a, -\alpha b, -\alpha^{-1}n).
\label{eq_9.26}
\end{equation}

The case $k=0$  $(\alpha = \infty)$.

Here $ u_{a,b}(x) = \prod_{i=1}^n x_i^a (1-x_i)^b;\; p=1$; and
$$
\Omega_{\lambda}(x) = m_{\lambda}(x)/m_{\lambda}(1_n).
$$

For each $\alpha\in \mathbb{N}^n$ define
$$
U_{\alpha}^{(a,b)}(x) = \prod_{i=1}^n G_{\alpha_i}^{(a,b)}(x_i),
$$
a product of Jacobi polynomials of 1 variable. The $U$'s will be pairwise orthogonal for the measure $u_{a,b}(x)\mathrm{d}x$ on $[0,1]^n$, but not symmetric; and $U_{\alpha}^{(a,b)}(x)$ has leading term $(a+1)_{\alpha}\,x^{\alpha}$.

It follows that
$$G_{\lambda}^{(a,b)}(x;\infty) = \frac{1}{n!}\sum_{\omega\in S_n}U_{\omega\lambda}^{(a,b)}(x)$$
with leading term $(a+1)_{\lambda}\,\Omega_{\lambda}(x;\infty)$.

From Rodrigues' formula it follows that
$$
U_{\alpha}^{(a,b)}(x) = u_{a,b}(x)^{-1}D^{\alpha}(x^{\alpha}(1-x)^{\alpha}u_{a,b}(x)),
$$
where $D^{\alpha} = \prod^n_{i=1}D_i^{\alpha_i}$, $D_i = \partial/\partial x_i$,
and hence that
$$
G_{\lambda}^{(a,b)}(x;\infty) = u_{a,b}(x)^{-1}\frac{1}{n!}\sum_{\omega\in S_n}D^{\omega\lambda}(x^{\omega\lambda}(1-x)^{\omega\lambda}u_{a,b}(x)).
$$
This suggests that in general we should define a differential operator $\Phi_{\lambda}$ as follows : if
$$\Omega_{\lambda}(x) = \sum_{\alpha}a_{\lambda\alpha}\,x^{\alpha}$$
then $$\Phi_{\lambda} = \sum_{\alpha}a_{\lambda\alpha}D^{\alpha}\circ(x^{\alpha}(1-x)^{\alpha})$$
i.e., $$\Phi_{\lambda}(f) = \sum_{\alpha}a_{\lambda\alpha}D^{\alpha}(x^{\alpha}(1-x)^{\alpha}f)$$
and a conjectured generalization of Rodrigues' formula:
$$
G_{\lambda}^{(a,b)}(x;\alpha) = u_{a,b}(x)^{-1}\Phi_{\lambda}(u_{a,b}(x)).
$$
This is not the right definition of $\Phi_{\lambda}$. We must replace the $\Omega_{\lambda}$ by the appropriate dual bases to make things work.
\bigskip

\subsection*{Additional observation}

For each standard tableau $T$ of shape $\lambda/\mu$: say
$$T:\quad \lambda = \lambda_0^{(0)}\supset\lambda_1^{(1)}\supset\cdots\supset\lambda_r^{(r)} = \mu \quad\quad (r=|\lambda - \mu|)$$
define
$$f_T(C) = \prod_{i=1}^r \left\{ {\left.\binom{\lambda^{(i-1)}}{\lambda^(i)}\right/\big(iC+2\rho(\lambda/\lambda^{(i)}})\big)\right\}.$$
Then
$$c_{\lambda/\mu}(C) = \sum_Tf_T(C),$$
summed over all standard tableaux $T$ of shape $\lambda/\mu.$

We have then
$$G_{\lambda}^{(a,b)}(x) = \sum_{\mu\subset\lambda}(-1)^{|\mu|}\frac{(a+p)_{\lambda}}{(a+p)_{\mu}}c_{\lambda/\mu}(a+b+2p)\Omega_{\mu}(x).$$
\newpage

\section*{Hermite polynomials} 

Here the measure is
$$e^{-p_2(x)}\Delta(x)^{2k}\mathrm{d}x = e(-x^2)\mathrm{d}\mu(x)$$
on $\mathbb{R}^n$, and the Hermite polynomials $H_{\lambda}(x;\alpha)$ will be eigenfunctions of the differential operator $E$ defined by
$$Ef = e(x^2)\Delta(x)^{-2k}\sum^n_{i=1}D_i(e(-x)^2\Delta(x)^{2k}D_if).$$
Explicitly we find
$$Ef = -2\sum^n_{i=1}x_iD_if+\sum^n_{i=1}D_i^2f + 2k\sum_{i\neq j}\frac{D_if}{x_i-x_j}$$
so that the eigenvalue is $-2|\lambda|$, i.e.,
$$EH_{\lambda} = -2|\lambda|H_{\lambda}$$
Let
$$\Box_2 = \frac{1}{2}\sum^n_{i=1}x^2_iD^2_i + k\sum_{i\neq j}\frac{x_i^2D_i}{x_i-x_j} - k(n-1)\sum x_i D_i$$
$$\varepsilon = \sum^n_{i=1}D_i$$
Then $$E = [\varepsilon,[\varepsilon,\Box_2]] - 2\sum x_iD_i$$
and hence
\begin{eqnarray*}
E\Omega_{\lambda} &=& (\varepsilon^2\Box_2 - 2\varepsilon\Box_2\varepsilon + \Box_2\varepsilon^2 - 2|\lambda|)\Omega_{\lambda} \\
&=& \sum_{\nu\subset\mu\subset\lambda}\binom{\lambda}{\mu}\binom{\mu}{\nu}(\rho(\lambda/\mu)-\rho(\mu/\nu))\Omega_{\nu} - 2|\lambda|\Omega_{\lambda}
\end{eqnarray*}
summed over $\nu\subset\mu\subset\lambda,\;\; |\lambda - \mu| = |\mu - \nu| = 1.$

So if $$H_{\lambda} = \sum_{\mu\subset\lambda}a_{\lambda\mu}\Omega_{\mu}$$
we have the recurrence relation for the coefficients $a_{\lambda\pi}$ :
$$-2|\lambda - \pi|a_{\lambda\pi} = \sum_{\mu\supset\nu\supset\pi}a_{\lambda\mu}\binom{\mu}{\nu}\binom{\nu}{\pi}(\rho(\mu/\nu) - \rho(\nu/\pi))$$
summed over $\mu,\nu$ such that $\lambda\supset\mu\supset\nu\supset\pi$ and $|\mu - \nu| = |\nu - \pi| = 1$.

We may take $a_{\lambda\lambda} = 1$, \& then it is clear from this recurrence relation that
\begin{equation}
a_{\lambda\pi} = 0 \mbox{   unless   } |\lambda - \pi| \mbox{   is \underline{even}}.
\tag{$\ast$}
\label{star}
\end{equation}

For $\alpha=2$ this is in James \cite{James75}.

We have then
$$\int_{\mathbb{R}^n}H_{\lambda}(x)H_{\mu}(x)e(-x^2)\mathrm{d}\mu(x) = 0$$
if $\lambda\neq\mu$, but the value of this integral when $\lambda=\mu$ still remains to be calculated (as in the case of the Jacobi polynomials).

From \eqref{star} it follows that
$$H_{\lambda}(-x) = (-1)^{|\lambda|}H_{\lambda}(x).$$

Let
$$ c_n = \int_{\mathbb{R}^n}e(-y^2)\mathrm{d}\mu(y) = \pi^{n/2}/2^{kn(n-1)/2}$$
from Selberg's integral (\& Stirling's formula). Put
$$F_{\lambda}(x;\alpha)=c_n^{-1}\int e(-y^2)\Omega_{\lambda}(x+iy;\alpha)\mathrm{d}\mu(y)$$
the range of integration being $\mathbb{R}^n$. Then the generating function for the polynomials $F_{\lambda}$ is
$$F(x,z) = \sum_{\lambda}(2\alpha)^{|\lambda|}F_{\lambda}(x)J_{\lambda}^*(z) = c_n^{-1}\int e(-y^2)e(2(x+iy),z)\mathrm{d}\mu(y).$$

Suppose that $\alpha = 2$ (or $1$, or $\frac{1}{2}$...). Since $F(x,z) = F(x,kzk')$,   $(k\in K= O(n))$, we have
\begin{eqnarray*}
F(x,z) &=& c_n^{-1}\int_{\Sigma}e^{-\tr(t^2)}\int_K e^{\tr(2(x+it)kzk')}\mathrm{d}k\mathrm{d}t \\
&=& c_n^{-1}\int_{\Sigma \times K} \exp -\ \tr((t-ikzk')^2+z^2-2xkzk')\mathrm{d}k\mathrm{d}t.
\end{eqnarray*}

If we put $s=t-ikzk'$ and integrate first with respect to $s$, we obtain
$$F(x,z) = \int_K \exp - \tr(z^2)\cdot \exp \tr(2xkzk')\mathrm{d}k$$
i.e.,
$$\boxed{F(x,z) = e(-z^2)e(2x,z).}$$
\smallskip

Since this holds for $\alpha = 2,1,\frac{1}{2}$ we may hope that it holds for all $\alpha$, i.e. that
$$\boxed{\sum_{\lambda}\alpha^{|\lambda|}F_{\lambda}(x)J_{\lambda}^*(z) = e(-z^2)e(k,z).}$$
Next consider
$$\int e(-x^2)F(x,z_1)F(x,z_2)\mathrm{d}\mu(x) = \int e(-x^2)e(-z_1^2)e(-z_2^2)e(2x,z_1)e(2x,z_2)\mathrm{d}\mu(x).$$

Again, when $\alpha = 2$ this is equal to
$$\int_{\Sigma} \exp - \tr(s^2+z_1^2+z_2^2)\int_{K\times K} \exp \tr (2sk_1z_1k_1'+2sk_2z_2k_2')\mathrm{d}k_1\mathrm{d}k_2\mathrm{d}s$$
in which the exponent is
$$-\tr(s^2+z^2_1+z^2_2-2s(k_1z_1k_1'+k_2z_2k_2')) = -\tr((s-k_1z_1k_1' - k_2z_2k_2')^2)+2\tr k_1z_1k_1'k_2z_2k_2'.$$
Integrate first w.r.t. $t=s-k_1z_1k'_1 - k_2z_2k_2'$; we obtain
$$c_n\int_{K\times K}e^{2\tr(z_1k_1'k_2z_2'k_1)}\mathrm{d}k_1\mathrm{d}k_2 = c_n e(2z_1,z_2)$$
so we have (when $\alpha = 2, 1,\frac{1}{2}$)
$$\boxed{\int e(-x^2)F(x,z_1)F(x,z_2)\mathrm{d}\mu(x) = c_n e(2z_1,z_2).}$$

Again let us hope that this formula is valid for all $\alpha$. Then we obtain
$$\sum_{\lambda,\mu}(2\alpha)^{|\lambda|+|\mu|}\left(\int F_{\lambda}(x)F_{\mu}(x)e(-x^2)\mathrm{d}\mu(x)\right) \cdot J_{\lambda}^*(z_1)J_{\mu}^*(z_2) = c_n\sum_{\lambda}(2\alpha)^{|\lambda|}\frac{J_{\lambda}^*(z_1)J_{\lambda}^*(z_2)}{J_{\lambda}^*(1_n)}$$
and hence that
$$\int F_{\lambda}(x)F_\mu(x)e(-x^2)\mathrm{d}\mu(x) = \frac{\delta_{\lambda\mu}c_n}{(2\alpha)^{|\lambda|}J_{\lambda}^*(1_n)}.$$

In other words the $F_{\lambda}$ are, up to a scalar factor, the Hermite polynomials. We normalize them as follows (for compatibility with the case $n=1$)
$$H_{\lambda}(x;\alpha) = 2^{|\lambda|}F_{\lambda}(x;\alpha) = \frac{2^{|\lambda|}}{c_n}\int\Omega_{\lambda}(x+iy)e(-y^2)\mathrm{d}\mu(y)$$
so that we have
$$\int H_{\lambda}(x)H_{\mu}(x)e(-x^2)\mathrm{d}\mu(x) = \frac{\delta_{\lambda\mu}c_n2^{|\lambda|}}{\alpha^{|\lambda|}J_{\lambda}^*(1_n)}$$
in which

\begin{eqnarray*}
\alpha^{|\lambda|}J_{\lambda}^*(1_n)
 &=& \frac{\alpha^{|\lambda|}J_{\lambda}(1_n)}{\langle J_{\lambda},J_{\lambda}\rangle _{\alpha}} \\
&=& \alpha^{|\lambda|}\prod_{s\in\lambda}\frac{n+\alpha a'(s) - l'(s)}{(\alpha a(s)+l(s)+1)(\alpha a(s)+l(s)+\alpha)}
\end{eqnarray*}

The generating function for the $H_{\lambda}$ is
$$
\sum_{\lambda}\alpha^{|\lambda|}H_{\lambda}(x)J_{\lambda}^*(y) = F(x,y) = e(2x,y)e(-y^2).
$$
When $\alpha = 0$, we get
$$\sum_{\lambda}\alpha^{|\lambda|}H_{\lambda}(0)J_{\lambda}^*(y) = e(-y^2)$$
so that $H_{\lambda}(0) = 0$ if $|\lambda|$ is odd, and
$$\sum_{|\lambda|=2m}\alpha^{|\lambda|}H_{\lambda}(0)J_{\lambda}^*(y) = \frac{(-1)^mp_2(y)^m}{m!}$$
from which it follows that
\begin{eqnarray*}
H_{\lambda}(0) &=& \frac{(-1)^m}{m!\alpha^{2m}}\langle J_{\lambda},p_2^m\rangle \\
&=& \frac{(-1)^m2^m}{\alpha^m} \times \mbox{coefficient of $p_2^m$ in $J_{\lambda}.$}
\end{eqnarray*}

But also
$$H_{\lambda}(0) = \frac{2^{|\lambda|}}{c_n}\int e(-y^2)\Omega_{\lambda}(iy)\mathrm{d}\mu(y)$$
so that
\begin{equation*}
\frac{1}{c_n}\int e(-y^2)\Omega_{\lambda}(y)\mathrm{d}\mu(y) =
\left\{
\begin{array}{lll}
\displaystyle
0, \mbox{ if $|\lambda|$ is odd}, \\
\\
(2\alpha)^{-m} \mbox{ (coeffs of $p_2^m$ in $J_{\lambda}$), if $|\lambda| = 2m.$}
\end{array}
\right.
\end{equation*}

So far, all this is proved only for $\alpha=2,1,\frac{1}{2}$.

\noindent\underline{Fourier transform}
\nopagebreak
\smallskip

Consider the integral
$$
J(y,z) = \int e(ix,y)e(-x^2/2)e(2x,z)e(-z^2)\mathrm{d}\mu(x).
$$
Again let us assume that $\alpha = 2$( or 1, or 1/2...). Then this integral is
$$\int_{\Sigma}\int_K\int_K \exp \tr(ixk_1yk_1' - \frac{1}{2}x^2+2xk_2zk_2'-z^2)\mathrm{d}x\,\mathrm{d}k_1\mathrm{d}k_2$$
in which the exponent is
$$-\frac{1}{2} \tr (x^2 - 2ixk_1yk_1' - 4xk_2zk_2' + 2z^2) = -\frac{1}{2} \tr \{ (x-ikyk_1' - 2k_2zk_2')^2 + y^2 - 2z^2 - 4ik_1yk_1'k_2zk_2'\}.$$
Hence
\begin{eqnarray*}
J(y,z) &=& e^{-\frac{1}{2}\tr y^2 + \tr z^2}\int_{K\times K}\left(\int_{\Sigma}e^{-\frac{1}{2}\tr (x-ik_1yk_1'-2k_2zk_2')^2}\mathrm{d}x\right) e^{2i\, \tr yk_1'k_2zk_2'k_1}\mathrm{d}k_1\mathrm{d}k_2 \\
&=& e(-y^2/2)e(z^2)\left(\int e(-x^2/2)\mathrm{d}\mu(x)\right)e(2iy,z)
\end{eqnarray*}
and 
\begin{eqnarray*}
\int e(-x^2/2)\mathrm{d}\mu(x) &=& 2^{np/2}c_n \\
&=& 2^{np/2}\pi^{n/2}/2^{kn(n-1)/2} \\
&=& (2\pi)^{n/2}
\end{eqnarray*}
so that we have
$$
J(y,z) = (2\pi)^{n/2} e(-y^2/2) e(2iy,z)e(z^2).
$$
Since
$$
e(2x,z)e(-z^2) = \sum_{\lambda}\alpha^{|\lambda|}H_{\lambda}(x)J_{\lambda}^*(z),
$$
it follows that
$$\sum_{\lambda}\alpha^{|\lambda|}\left(\int e(ix,y)e(-x^2/2)H_{\lambda}(x)\mathrm{d}\mu(x)\right)J_{\lambda}^*(z) = (2\pi)^{n/2}e(-y^2/2)\sum_{\lambda}\alpha^{|\lambda|}H_{\lambda}(y)J_{\lambda}^*(iz)$$
and hence we have the \underline{Fourier transform formula}
$$\boxed{\int e(ix,y)e(-x^2/2)H_{\lambda}(x)\mathrm{d}\mu(x) = i^{|\lambda|}(2\pi)^{n/2}e(-y^2/2)H_{\lambda}(y).}$$
Again one hopes that this will be true for all values of $\alpha$.

(Since this is linear in $H_{\lambda}$, it will hold for all symmetric polynomials.)

\newpage
\bibliographystyle{plain}
\bibliography{mhbiblio}

\end{document}